\documentclass[a4paper]{article}
\usepackage[english]{babel}
\usepackage[utf8]{inputenc}
\usepackage{amsmath,amsthm}
\usepackage{amsfonts,amssymb}

\usepackage{mathtools}
\usepackage{enumitem}

\usepackage{url} 

\usepackage{tikz,tikz-cd} 
 \usetikzlibrary{decorations.pathmorphing} 

\usepackage{xifthen} 

\newcommand{\Kcal}{\mathcal{K}}

\newcommand{\I}{\textrm{I}}
\newcommand{\II}{\textrm{II}}

\newcommand{\corner}[1]{\,\tikz[baseline=(todotted.base)]{
        \node[inner sep=1pt,outer sep=0pt] (todotted) {$#1$};
        \draw (todotted.north west) -- (todotted.north east);
        \draw (todotted.south west) -- (todotted.north west);
    }}

\usepackage{titlesec}
\titleformat{\subsubsection}[runin]{\normalfont}{\thesubsubsection}{0pt}{}[.]

\renewcommand{\thesubsubsection}{\arabic{section}.\arabic{subsection}.\arabic{subsubsection}}
\csname @addtoreset\endcsname{subsubsection}{section}

\swapnumbers

\theoremstyle{plain}
\newtheorem{theorem}{Theorem}[subsection]
\newtheorem{lemma}[theorem]{Lemma}
\newtheorem{prop}[theorem]{Proposition}
\newtheorem*{prop*}{Proposition}
\newtheorem{cor}[theorem]{Corollary}

\theoremstyle{definition}
\newtheorem{definition}[theorem]{Definition}
\newtheorem{remark}[theorem]{Remark}
\newtheorem{notation}[theorem]{Notation}

\theoremstyle{remark}

\newtheorem{npar}[theorem]{\unskip}

\DeclareMathOperator{\Sd}{\texttt{Sd}}
\DeclareMathOperator{\Ex}{\texttt{Ex}}

\DeclareMathOperator*{\colim}{Colim}
\DeclareMathOperator*{\fprod}{\times}
\DeclareMathOperator{\Hom}{Hom}

\begin{document}

\pagestyle{plain}
\title{A constructive account of the Kan-Quillen model structure and of Kan's $\Ex^{\infty}$ functor}
\date{}

\author{Simon Henry}

\maketitle

\begin{abstract}
We give a fully constructive proof that there is a proper cartesian $\omega$-combinatorial model structure on the category of simplicial sets, whose generating cofibrations and trivial cofibrations are the usual boundary inclusion and horn inclusion. The main difference with classical mathematics is that constructively not all monomorphisms are cofibrations (only those satisfying some decidability conditions) and not every object is cofibrant.

The proof relies on three main ingredients: First, our construction of a weak model categories on simplicial sets, then the interplay with the semi-simplicial versions of this weak model structure and finally, the use of Kan $\Ex^{\infty}$-functor, and more precisely of S.Moss' direct proof that the natural map $X \rightarrow \Ex^{\infty} X$ is an anodyne morphism, which we show is constructive when $X$ is cofibrant.
\end{abstract}

\renewcommand{\thefootnote}{\fnsymbol{footnote}} 
\footnotetext{\emph{Keywords.} Model categories, constructive mathematics, simplicial sets, $\Ex^{\infty}$-functor.}
\footnotetext{\emph{2010 Mathematics Subject Classification.} 55U35,55U40,18G30.}
\footnotetext{This work was supported by the Operational Programme Research, Development and Education Project ``Postdoc@MUNI'' (No. CZ.02.2.69/0.0/0.0/16\_027/0008360)}
\renewcommand{\thefootnote}{\arabic{footnote}}



\tableofcontents

\makeatletter 
\renewcommand{\thetheorem}{\thesection.\arabic{theorem}} 
\@addtoreset{theorem}{section}
\makeatother

\section{Introduction}

The goal of this paper is to give a fully constructive proof of the existence of the usual Kan-Quillen model structure on simplicial sets, and of some of its classical properties. Constructive can be taken to mean ``Without the axiom of choices and the law of excluded middle'', or a bit more precisely as ``in the internal logic of an elementary topos with a natural number objects''. It can also be formalized in Aczel's (CZF) \cite{Aczel2010CST} and probably in considerably weaker foundation as well, see remark \ref{Rk:predicativity}.

Our main theorem is:

\begin{theorem} \label{main_theorem}

There is a proper cartesian Quillen model structure on the category of simplicial sets such that:

\begin{itemize}

\item The trivial fibrations are the morphisms with the right lifting property against all boundary inclusion $\partial \Delta [n] \hookrightarrow \Delta[n]$.

\item Cofibrations are the monomorphisms $f:A \rightarrow B$ which are ``level wise complemented''  (i.e. for all integer $n$ for each $b \in B([n])$ it is decidable if $b \in A([n])$ or not), and such that for all $b \in B([n]) - A([n])$ it is decidable if $b$ is degenerated or not. 

\item The fibrations are the ``Kan fibrations'', i.e. they are the morphisms with the right lifting property against the horn inclusion: $\Lambda^k[n] \hookrightarrow \Delta[n]$. Dually trivial cofibrations are the retract of $\omega$-transfinite composition of pushout of coproduct of horn inclusions (called anodyne morphisms).

\end{itemize}

\end{theorem}

Note that assuming the law of excluded middle the class of cofibrations boils down to the class of all monomorphisms and hence one recovers the usual Kan-Quillen model structure.

After we announced this result, two other proofs, relying on different tools, have been found by N.Gambino, C.Sattler and K.Szumilo and should appears soon.

\bigskip

This theorem is obtained by patching together the following results: theorem \ref{th:Simplicial_Quillen_MS_1} gives the existence of a model structure with the appropriate cofibrations and trivial fibrations, proposition \ref{prop:leftproper} gives left properness, proposition \ref{prop:Kan=Strong} shows that the fibrations and trivial cofibrations are indeed as specified here and proposition \ref{prop:rightProperness} shows that it is also right proper. Cartesianess was already known, but reproved in \ref{prop:pushoutProdCond}.

\bigskip

One can also say a few word about the equivalences of the model structure of \ref{main_theorem}: they are defined (in \ref{def:Simplicial_weak_equivalence}) using the forgetful functor to semi-simplicial sets and the weak model structure on semi-simplicial sets constructed in \cite[theorem 5.5.6]{henry2018weakmodel}. Proposition \ref{prop:Simp_equiv_vs_SemiSimp_equiv} shows that this notion of equivalence is compatible to the notion we used in \cite{henry2018weakmodel}, in particular \cite[proposition 5.2.5]{henry2018weakmodel} shows that they admit the usual characterization in terms of homotopy groups, as long as the homotopy groups are defined not as quotient sets but as \emph{setoids}.

\bigskip

As we do not assume the axiom of choice, one needs to precise some details regarding theorem \ref{main_theorem}: a ``structure of fibration'' (resp. trivial fibration) on a map $f$ is the choice of a solution to each lifting problem of a horn inclusion (resp. boundary inclusion) against $f$. No uniformity condition is required on these lift. A fibration (resp. trivial fibration) is a morphism which admits at least one structure of fibration (resp. trivial fibration), but the choice of the structure is considered irrelevant.

More generally, we will follow the convention that (unless exceptionally stated otherwise) every statement of the form $\forall a, \exists b$ should be interpreted as the existence of a function that given ``$a$'' produces a ``$b$''. In particular, when one says that a morphism has the lifting property against some set of arrow it means that one has a function that produces a solution to each lifting problem. We will use the convention constantly in the present paper, i.e. every time we say that ``there exists'' some $x$, we mean that one specific $x$ has been chosen for each possible value of the parameters involved in the statement.

\bigskip

As fibrations and trivial fibrations are defined by the right lifting property against a small set of morphisms between finitely presented objects, it is very easy to apply a constructive version of the small object argument to show that one has two weak factorization systems, which will be called as follows:

\begin{definition} \label{def:kanfib-anodyne-cof-fibtriv}
\begin{itemize}

\item[]

\item The weak factorization system cofibrantly generated by the boundary inclusion $\partial \Delta[n] \hookrightarrow \Delta[n]$ is called ``cofibrations/trivial fibrations.

\item The weak factorization system cofibrantly generated by the horn inclusion $\Lambda^k[n] \hookrightarrow \Delta[n]$ is called ``Anodyne morphisms/Kan fibrations''.

\end{itemize}
\end{definition}

We have discussed the constructive validity of the small object argument in appendix D of \cite{henry2018weakmodel}, though there are probably other references doing this.

Note that anodyne morphisms will in the end be the trivial cofibrations, and Kan fibrations will be what we have called fibrations in the statement of the main theorem \ref{main_theorem}, but this will be one of the last result we will prove. In the meantime we will distinguish between Kan fibrations and ``strong fibrations'' and between anodyne morphisms and ``trivial cofibrations'' (these two other concept being defined in \ref{def:Simplicial_weak_equivalence}). Simplicial sets whose map to the terminal simplicial set is a Kan fibration will be called either Kan complexes, or fibrant simplicial sets.

\begin{remark} Before going any further, we should pause here to insist on a very important remark: one of the key difference between what we are doing in the present paper and the usual construction of the Kan-Quillen model structure in classical mathematics is that the cofibrations are no longer exactly the monomorphisms. It can be shown, see for example proposition 5.1.4 in \cite{henry2018weakmodel}, that the class of cofibrations generated by the boundary inclusion, i.e. the class of arrow which have the left lifting property against all trivial fibration is exactly the class of cofibrations described in the statement of theorem \ref{main_theorem}. In particular one has: 

\begin{center}
\fbox{ \parbox{8cm}{\vspace{-0.3cm}\begin{center}\textbf{Not every simplicial set is cofibrant !} \\ A simplicial set $X$ is cofibrant if and only if it is decidable whether a cell of $X$ is degenerate or not.\end{center}\vspace{-0.3cm}}}
\end{center}

This introduces some changes compared to the classical situation, for example the left properness of the model structure on simplicial set is no longer automatic, and the assumption that certain objects needs to be cofibrant tends to appears in a lot of results. Compare for example \ref{cor:Ex_infty_fibrant}, \ref{prop:Ex_infty_Preserves_Fib} or \ref{prop_XtoExyXisAnodyne} to their classical counterparts. 

One can also show the classical Eilenberg-Zilber lemma, asserting that a cell $x \in X([n])$ can be written uniquely as $\sigma^*y$ for $\sigma$ a degeneracy and $y$ a non-degenerated cells holds if and only if $X$ is cofibrant. A general constructive version of the Eilenberg-Zilber lemma can be found as Lemma 5.1.2 in \cite{henry2018weakmodel} and does implies that the statement above holds for cofibrant simplicial sets. The converse (that the validity of the Eilenberg-Zilber lemma implies cofibrancy of $X$) is immediate from the decidability of equality between morphisms of the category $\Delta$: if a cell is written $\sigma^* y $ with $y$ non-degenerate one can decide if it is degenerate or not depending on if $\sigma$ is the the identity (an isomorphism)  or not.
\end{remark}

\bigskip

The general structure of the proof of this theorem (and in fact of the paper) is as follows:

\begin{itemize}

\item In subsection \ref{subsection:WMS} we review the existence of a  ``weak model structure'' on simplicial sets and semi-simplicial sets from \cite{henry2018weakmodel}, which is our starting point.

\item In subsection \ref{subsec:SimpMS}, more precisely in theorem \ref{th:Simplicial_Quillen_MS_1}, we will (up to a technical detail, see the remark~\ref{rk:gap_in_the_proof} below) extend this in a model structure on the category of simplicial sets with cofibrations (and trivial fibration) as specified above, but we will not show that trivial cofibrations are the same as anodyne morphisms, or equivalently that the fibrations (called ``strong fibrations'') are the Kan fibrations.  This part is based on the use of semi-simplicial sets.

\item Left properness of this model structure follow also from semi-simplicial techniques (see proposition \ref{prop:leftproper}).

\item The overall goal\footnote{We will give a more detailed account of its contents at the beginning of this section.} of section \ref{sec:KanEx} is to introduce Kan's $\Ex^{\infty}$-functor. This is done following the work of S.Moss from \cite{moss2015another}, which can be made constructive at the cost of only minor modification. This will allows us to show that the fibrations of the model structure above are exactly the Kan fibrations (proposition \ref{prop:Kan=Strong}) and to prove the right properness of this model structure (proposition \ref{prop:rightProperness}), as well as to fix a small gap in constructiveness of subsection \ref{subsec:SimpMS} (see the remark below).

\end{itemize}

\begin{remark} \label{rk:gap_in_the_proof}
The gap we are referring too in this last point is that in subsection \ref{subsec:SimpMS}, the ``strong fibrations'' (i.e. the fibrations of the model structure on simplicial sets) are defined as the map having the right lifting property against all cofibrations which are equivalences. It is unclear if they can be defined by a lifting property against a small set and hence if trivial cofibration/strong cofibration do form a weak factorization system as a model category structure should require. In proposition \ref{prop:Jiri_WFS} we give a formal argument that shows it is the case, but it is unlikely that this argument can be made constructive. What definitely solve the problem constructively is the proof in \ref{prop:Kan=Strong} that this factorization is actually just the ``anodyne morphisms/Kan fibrations'' factorization, but this require all the material of section \ref{sec:KanEx}. 

This being said, the reader should note that even before section \ref{sec:KanEx}, it holds constructively that the anodyne/Kan fibration of an arrow with fibrant target is a ``trivial cofibration/strong fibration'' factorization (because of the third point of lemma \ref{Lem:Cof+eq=trivCof_and_=anodyne_if_fibranttarget}). Hence it holds constructively, even without the results of section \ref{sec:KanEx}, that any arrow with fibrant target admit such a factorization, i.e. one already has something similar to a (right\footnote{not exactly though: in the standard terminology a right semi-model structure concern a weakening of the cofibration/trivial fibration factorization to arrows with fibrant target, where here it is the other weak factorization system which is concerned.}) semi-model category without invoking the properties of Kan $\Ex^{\infty}$ functor.
\end{remark}

\begin{remark}

The fact that we need to invoke the good properties of Kan's $\Ex^{\infty}$ functor to show that the class of fibration is indeed the class of Kan fibrations of course remind us of D-C.Cisinski's approach to the construction of Kan-Quillen model structure in \cite{cisinski2006prefaisceaux}. We do not really know how deep are the similarities between our proof and D-C.Cisinski's proof.  Our initial plan on this problem was actually to try to see if this approach of Cisinski can be made constructive or not. While we definitely do not exclude this is the case, it seemed to represent a considerably harder task than what we have achieved here. One of the problem is that Cisinski's theory relies heavily on a set theoretical argument similar to the one we mention in the proof of \ref{prop:Jiri_WFS}, whose constructiveness seems unlikely. The other problem being simply that Cisinski approach, while very elegant, relies on a considerable amount of machinery whose constructivity would have to be carefully checked.
\end{remark}

\begin{remark} \label{Rk:predicativity} Finally, I only said that ``constructive'' meant something like internal logic of an elementary topos with a natural number object for simplicity, but everything is actually completely predicative for some, relatively strong, sense of this word.
I believe that everything can be formalized within the internal logic of an ``Arithmetic universe'', i.e. a pretopos with parametrized list objects (see for example \cite{maietti2012induction} ). Such a formalization of course requires some modification: for example it wouldn't make sense to say that a morphisms ``is a fibration'' in the sense that ``there exists a structure of fibrations on the morphisms'' as the set of all ``structure of fibration'' on a given morphism cannot be defined, but it would make sense to consider a morphism endowed with a structure of fibration, and to show that given such a pair one can perform some construction.

Though working in such framework in an explicitly way forces to be extremely careful about a huge number of details and makes everything considerably more complicated, and would make the paper considerably more complicated. For this reason we will not do it explicitly. It seems to me that this is typically the sort of thing that should be done with a proof assistant.

There is one part of this claim that I have not checked carefully: Whether such a weak framework is sufficient to use the case of the small object argument that we need, i.e. construct the cofibration/trivial fibration and the anodyne/Kan fibrations factorization systems (generated respectively by boundary inclusion and horn inclusion) on simplicial set and semi-simplicial sets, though it seems reasonable that a complicated encoding with list object can achieve this. More precisely this should follow from the fact that the initial model theorem for partial horn theories of Vickers and Palmgren in \cite{palmgren2007partial} is believed to be provable internally in an arithmetic universe, and the factorization obtained from R.Garner's version of the small object argument (from \cite{garner2009understanding}) are constructed as certain initial structure that can be described using partial horn logic.
\end{remark}

\begin{remark} In a joint paper with Nicola Gambino (\cite{Gambino2019simplicial}), we will show that this Quillen model structure on simplicial sets admit all the necessary structure to interpret homotopy type theory, with type and context being interpreted as bifibrant objects. This was the main motivation for the present paper and the two papers have been written in close connection. I would also like to thanks Nicola Gambino for the helpful comments he made about earlier version of the present paper.
\end{remark}

\begin{notation} $\Delta$ and $\Delta_+$ denotes the category of finite non-empty ordinal, respectively with non-decreasing map and non-decreasing injection between them. $\widehat{\Delta}$ is the category of simplicial sets and $\widehat{\Delta_+}$ is the category of semi-simplicial sets (see \ref{claim:SemiSimplicialWMS_exists}). One denotes by $\Delta[n]$ and $\Delta_+[n]$ the representable simplicial and semi-simplicial sets corresponding to the ordinal $[n]=\{0,\dots,n\}$. Our usual notation for the boundary of the $n$-simplex and its $k$-th horn, both for simplicial and semi-simplicial versions are: $ \partial \Delta[n] \quad  \Lambda^k[n] \quad \partial \Delta_+ [n] \quad \Lambda^k_+[n] $

The boundary inclusion map is denotes $\partial_n$ or $\partial [n] : \partial \Delta[n] \rightarrow \Delta[n]$, the $i$-th face maps is denoted $\partial^i[n]$ or $\partial^i_n$ or just $\partial^i : \Delta[n-1] \rightarrow \Delta[n]$, for the map corresponding to the order preserving injection from $[n-1]$ to $[n]$ which only skip $i$. The degeneracy $\Delta[n+1] \rightarrow \Delta[n]$ that hits $i$ twice is denoted $\sigma^i$.

Given a simplicial or semi-simplicial sets $X$, the image of a cell $x \in X_n$ be the $i$-th face map is denoted $d_i x$.

\end{notation}

\makeatletter 
\renewcommand{\thetheorem}{\thesubsection.\arabic{theorem}} 
\makeatother

\section{Constructing the model structure}

\subsection{Review of the weak model structures}
\label{subsection:WMS}

\begin{npar} \label{claim:SimplicialWMS_exists}
One of the achievement of \cite{henry2018weakmodel}, which is the starting point of the present paper, is the construction of a ``weak model structure'' on the category of simplicial sets where fibrations (between fibrant objects) and cofibrations (between cofibrant objects) are as specified above.

More explicitly this means that there is a class of maps called ``equivalences\footnote{In most of the literature this are called weak equivalence, though we can't think of any reasons to keep the adjective ``weak'' other than history, so we will simply drop it.}'' in the category of simplicial sets that are either fibrant or cofibrant (in the sense above) such that:

\begin{itemize}

\item Weak equivalences (between objects that are either fibrants or cofibrant) contains isomorphisms, are stable under composition and satisfies $2$-out-of-$3$ (and the stronger $2$-out-of-$6$ property).

\item A cofibration between cofibrant objects is a weak equivalence if and only if it has the left lifting properties against all fibrations between fibrant objects (such a map is called a trivial cofibration).

\item A fibrations between fibrant objects is a trivial fibrations if and only if it is a weak equivalence\footnote{Here we use the fact that that trivial fibrations are characterized by a lifting property against cofibration between cofibrant objects, which might not be the case in a general weak model category.}.

\item The localization of the category of fibrant or cofibrant objects at the weak equivalences can be described as the category of fibrant and cofibrant objects with homotopy classes of maps between them. Where the homotopy relation is defined as usually, using equivalently a path object or a cylinder object. This localization is called the homotopy category.

\item The weak equivalences are exactly the morphisms that are invertible in the homotopy category (which proves the first point immediately).

\end{itemize}

One can deduce from this various characterization of weak equivalences: for example, a map from a cofibrant object to a fibrant object is a weak equivalence if and only if it can be factored as a trivial cofibration followed by a trivial fibration. Note that at this point it does not makes sense to ask whether a map $X \rightarrow Y$ is a weak equivalence if one of $X$ or $Y$ is neither fibrant nor cofibrant. 

\end{npar}

\begin{npar} \label{claim:SemiSimplicialWMS_exists} In \cite[theorem 5.5.6]{henry2018weakmodel} we also showed that a similar ``weak model structure'' exists on the category of semi-simplicial sets. Semi-simplicial sets are ``simplicial sets without degeneracies'', i.e. collection of sets $X_0,\dots,X_n,\dots$ with ``faces maps'' satisfying the same relations as the face maps of a simplicial sets. Equivalently they are presheaves on the category $\Delta_+$ of finite non-empty ordinals and injective order preserving maps between them.

The generating cofibrations in the category of semi-simplicial sets are the semi-simplicial boundary inclusion:

 \[ \partial \Delta_+[n] \hookrightarrow \Delta_+[n], \]

where $ \partial \Delta_+[n]$ and $\Delta_+[n]$ respectively denotes the semi-simplicial subset of non-degenerated cells in $\Delta[n]$ and $\partial \Delta[n]$. Note that the $\Delta_+[n]$ also corresponds to the representable semi-simplicial sets, so that a morphism $\Delta_+[n] \rightarrow X$ is the same as an $n$-cell of $X$ and a morphism $\partial \Delta_+[n] \rightarrow X$ is the data of a collection of $n$ cells of dimension $n-1$ with compatible boundary exactly as simplicial morphisms from $\partial \Delta[n]$ to a simplicial sets $X$. In particular a morphism $f :X \rightarrow Y$ of simplicial sets is a trivial fibration if and only if its image by the forgetful functor to semi-simplicial sets is a trivial fibration (in the sense that it has the right lifting property against the generating cofibration).

As there is no degeneracies anymore in $\widehat{\Delta_+}$ the description of cofibrations simplifies to just ``levelwise complemented monomorphisms''  i.e. the class of monomorphism $f :X \rightarrow Y$ such that for each $n$, and for each $ y \in Y([n])$ it is decidable whether $y \in X([n])$ or not (this is also discussed in \cite[theorem 5.5.6]{henry2018weakmodel}). In particular, every object is cofibrant.

Similarly, a morphism of semi-simplicial sets is said to be a Kan fibration when it has the lifting property against the semi-simplicial version of the horn inclusion $\Lambda^k_+[n] \hookrightarrow \Delta_+[n]$, where $\Lambda^k_+[n]$ and $\Delta_+[n]$ respectively denotes respectively the semi-simplicial sets of non-degenerate cells in $\Lambda^k[n]$ and $\Delta[n]$). As above a simplicial morphisms between simplicial sets is a Kan fibration if and only if its image by the forgetful functor to simplicial set is a Kan fibration of semi-simplicial sets.

In this weak model structure on semi-simplicial sets, the cofibration are as described above, the fibrant objects are the semi-simplicial Kan complexes and the fibrations and trivial fibration between fibrant object are the Kan fibrations and trivial fibrations. The big difference with the model structure on simplicial sets is that as every semi-simplicial set is cofibrant, the classes of weak equivalences is defined between arbitrary objects of the category.
Note that we do not claim that every trivial cofibration (i.e. cofibration which is an equivalence) is an anodyne morphism (i.e. a retract of a transfinite composition of pushout of coproducts of semi-simplicial horn inclusion) : the anodyne morphism have the left lifting property against all Kan fibrations, the trivial cofibration only against Kan fibration between Kan complexes.

\end{npar}

\begin{remark} \label{Rk:SemiSimpNotQuillen}Note that it is well known, even classically, that this model structure cannot be a Quillen model structure. As every object is cofibrant, it can be seen by a combinatorial argument that, at least classically, it is a ``right semi-model structure'' in the sense of \cite{barwick2010left}). But for example the codiagonal map $\Delta_+[0] \coprod \Delta_+[0] \rightarrow \Delta_+[0]$, where $\Delta_+[0]$ denotes the representable semi-simplicial sets by the ordinal $[0]= \{0\}$ is easily seen to have the lifting property of trivial fibrations (there is no higher cells to lift ! ) while it is clearly not a weak equivalence. \end{remark}

 The forgetful functor from simplicial sets to semi-simplicial sets is very well behaved: we showed in \cite[theorem 5.5.6]{henry2018weakmodel} that it is both a left and right Quillen equivalence, and we will prove in \ref{prop:Simp_equiv_vs_SemiSimp_equiv} that it preserves and detect weak equivalences without any assumption of fibrancy/cofibrancy. As all object in $\widehat{\Delta_+}$ are cofibrant, this will allow to remove some assumption of cofibrancy in various places.

\begin{proof}[Sketch of proof of \ref{claim:SimplicialWMS_exists}] We finish this section by presenting the main steps of the argument given in \cite{henry2018weakmodel} of the existence of the weak model structure on simplicial sets, i.e. all the claims made in \ref{claim:SimplicialWMS_exists}. The details of this can be found in \cite{henry2018weakmodel}, but we hope the following summary will be of help to the reader. The proof for semi-simplicial sets is quite similar.

The first (and essentially only) important technical step is the proof of the so-called ``pushout-product'' or ``corner-product'' conditions for the simplicial generating cofibrations and trivial cofibrations. This follows from a completely constructive results of Joyal (theorem 3.2.2 of \cite{joyal2008notes}), in \cite{henry2018weakmodel} it corresponds to lemma 5.2.2 (and how it is used in the proof of theorem 5.2.1 in 5.2.3). In the present paper we also reproduce a different proof of this claim as \ref{prop:pushoutProdCond}, which is due to S.Moss (see \cite[2.12]{moss2015another}).

From the corner-product conditions one deduces formally\footnote{using the so-called ``Joyal-Tierney calculus'' presented in the appendix of \cite{joyal2006quasi}, though this types of manipulation were known before, maybe in a less elegant or general way.} all the usual property of stability of cofibrations, anodyne morphisms, fibrations, and trivial fibrations under product and exponential expected in a cartesian model category (see proposition \ref{prop:pushoutProdCond} and the remark directly below it).

This allows to construct nicely behaved cylinder objects as $\Delta[1] \times X$ and path objects as $X^{\Delta[1]}$, whose legs are appropriately (trivial) (co)fibrations as soon as $X$ is (co)fibrant. More generally, one can construct relative path object for any fibration $X \twoheadrightarrow Y$ and relative cylinder object for any cofibration $A \hookrightarrow Y$. Having such relative cylinders and path objects is the definition of weak model structure that we gave in section 2 of \cite{henry2018weakmodel}. The precise observation that one get a weak model structure from such a tensor product satisfying the corner-product condition is essentially the construction done in section 3 of \cite{henry2018weakmodel}, summarized by theorem 3.2 there.

Then all the claims made in \ref{claim:SimplicialWMS_exists} follows from the general theory of weak model structure developed in section 2.1 and 2.2 of \cite{henry2018weakmodel}. We sketch the general strategy here, though at this point we recommend looking directly at subsection 2.1 and 2.2 of \cite{henry2018weakmodel} which are mostly self contained.

One uses these cylinders and path objects to define the homotopy relation between maps from a cofibrant object to a fibrant objects. Using the lifting property one show that the homotopy relation with respect to any cylinder object is equivalent to the homotopy relation with respect to any path object and that these define an equivalence relation compatible to pre-composition and post-composition. The proof is essentially the same as in a full Quillen model structure: the definition of weak model structure is exactly tailored so that the usual proof of these claims can be applied.

This allows to give a first definition of the homotopy category as the category whose objects are the fibrant-cofibrant objects and the maps are the homotopy class of maps. One then proves formally that this homotopy category is equivalent to various localization (see theorem 2.2.6 in \cite{henry2018weakmodel}), the last one being the localization of the category of simplicial sets that are either fibrant or cofibrant at all trivial cofibration with cofibrant domain and all trivial fibration with fibrant target. One can then defines weak equivalences as the arrow that are invertible in this localization, and one automatically have $2$-out-of-$6$ and all the other good properties of weak equivalences. The fact that trivial fibration (with fibrant domain) are exactly the fibration that are equivalence is a little harder and use again the property of the relative path objects (see proposition 2.2.9 in \cite{henry2018weakmodel}), and similarly for cofibrations.\end{proof}

\subsection{The simplicial model structure}
\label{subsec:SimpMS}

To obtain that simplicial sets form a full Quillen model structure we first need to extend the meaning of ``equivalences'' so that it makes sense also for arrows between objects that are neither fibrant nor cofibrant.  We will do this by exploiting the forgetful functor from the category of simplicial sets to the category $\widehat{\Delta_+}$ of semi-simplicial sets. As in the category of semi-simplicial sets every object is cofibrant the notion of weak equivalence there is defined for arbitrary arrows, and we will show it is reasonable to define equivalences of simplicial sets as arrow that are equivalences of the underlying semi-simplicial sets.

We start by the following observation:

\begin{lemma} \label{lem:TrivFib_anodyne_Are_Eq_in_Delta+}

\begin{itemize}

\item[]

\item If $f:X \rightarrow Y$ is an anodyne morphism in $\widehat{\Delta}$, then its image in $\widehat{\Delta_+}$ is also anodyne, and in particular is an equivalence.

\item Let $f:X \twoheadrightarrow Y$ be a trivial fibration in $\widehat{\Delta}$. Then the image of $f$ in $\widehat{\Delta_+}$ is an equivalence.

\end{itemize}
\end{lemma}

Note that in the second case, it is obvious that $f$ is a trivial fibration in $\widehat{\Delta_+}$, but this is not enough to deduce that is is an equivalence in general, unless its target is fibrant, as $\widehat{\Delta_+}$ only has a weak model structure.

\begin{proof}
\begin{itemize}

\item[]

\item This is corollary 5.5.15.(ii) of \cite{henry2018weakmodel}.

\item We first assume that $X$ is cofibrant. In this case one can construct a strong cylinder object for $X$ using the cartesian structure of simplicial sets: 

\[ X \coprod X \hookrightarrow \Delta[1] \times X  \overset{\sim}{\rightarrow} X \]

with the two maps $X \hookrightarrow \Delta[1] \times X$ being anodyne morphisms (this follows from the fact that $X$ is cofibrant and the corner-product conditions). Because of the previous point this produces a strong cylinder object for the underlying semi-simplicial set of $X$ in the category of semi-simplicial sets.

In $\widehat{\Delta_+}$, every object is cofibrant, and the arrow $f:X \rightarrow Y$ is still a trivial fibration, so one can find some dotted lifting for the following two squares in $\widehat{\Delta_+}$:

\[\begin{tikzcd}
\emptyset \ar[d,hook] \ar[r] & X \ar[d,two heads] \\
Y \ar[r] \ar[ur,dotted,"s"] & Y\\
\end{tikzcd} \qquad 
\begin{tikzcd}
X \coprod X \ar[d,hook]  \ar[rr,"{(Id_X , s f)}"] & & X \ar[d,two heads]  \\
\Delta[1] \times X \ar[r] \ar[urr,dotted,"h"] & X \ar[r] & Y \\ 
\end{tikzcd}
\]

In particular, $s$ is a section of $f$, i.e. $f s =Id_Y$, and $h$ an homotopy between $Id_X$ and $sf$. Hence $s$ is an inverse of $f$ in the homotopy category of $\widehat{\Delta_+}$, which makes $f$ an equivalence in $\widehat{\Delta_+}$.

In the general case (when we do not assume that $X$ is cofibrant), one take a cofibrant replacement (with a trivial cofibration) $X^c \overset{\sim}{\twoheadrightarrow} X$ and the result above applies to both the trivial fibration $X^c \overset{\sim}{\twoheadrightarrow} X$ and the composite trivial fibration $X^c \overset{\sim}{\twoheadrightarrow} Y$. By $2$-out-of-$3$ for weak equivalences in $\widehat{\Delta_+}$ this implies that the map $X \overset{\sim}{\twoheadrightarrow} Y$ is indeed an equivalence in $\widehat{\Delta_+}$.
\end{itemize}
\end{proof}

\begin{prop} \label{prop:Simp_equiv_vs_SemiSimp_equiv}
For a morphism $f:X \rightarrow Y$ between simplicial sets that are either fibrant or cofibrant the following are equivalent:

\begin{itemize}

\item $f$ is an equivalence for the weak model structure in $\widehat{\Delta}$.

\item The image of $f$ in $\widehat{\Delta_+}$ is an equivalence for the weak model structure on $\widehat{\Delta_+}$

\end{itemize}

\end{prop}

\begin{proof}

If $Y$ is cofibrant, then one can take a fibrant replacement $Y \overset{\sim}{\hookrightarrow} Y^f$. The map $Y  \overset{\sim}{\hookrightarrow} Y^f $ is an equivalence both in $\widehat{\Delta}$ and $\widehat{\Delta_+}$, so in both category $f$ is an equivalence if and only if the composite $X \rightarrow Y^f$ is an equivalence, so it is enough to prove the result when $Y$ is fibrant. A similar argument using a cofibrant replacement allows to assume that $X$ is cofibrant.

Assuming both $X$ cofibrant and $Y$ fibrant, one factors $f$ as an anodyne morphism (with cofibrant domain) followed by a Kan fibration (with fibrant target). The anodyne morphism is an equivalence in both categories, hence (in both category) $f$ is an equivalence if and only if the Kan fibration part is a trivial fibration. But for a map in $\widehat{\Delta}$, being a trivial fibration in $\widehat{\Delta}$ and in $\widehat{\Delta_+}$ are the exact same condition (the lifting property only involves face operation, no degeneracies).

\end{proof}

This last proposition makes the following definition very reasonable:

\begin{definition} \label{def:Simplicial_weak_equivalence}

\begin{itemize}

\item[]

\item An arrow in $\widehat{\Delta}$ is said to be an \emph{equivalence} if its image by the forgetful functor to $\widehat{\Delta_+}$ is an equivalence for the semi-simplicial version of the Kan-Quillen weak model structure mentioned in \ref{claim:SemiSimplicialWMS_exists}.

\item A trivial cofibration is a cofibration which is also an equivalence.

\item A strong fibration is an arrow that has the right lifting property against all trivial cofibrations.

\end{itemize}
\end{definition}

We remind that the reader, that we will prove in \ref{prop:Kan=Strong} that these notion of strong fibrations and trivial cofibrations are equivalent to the usual notion of Kan fibrations and anodyne morphisms.

\begin{remark} \label{rk:def_WE_fibrations}With this definition it is immediate that: 

\begin{itemize}

\item Isomorphisms are equivalences, and equivalences are stable under composition, satisfies the $2$-out-of-$3$ and even the $2$-out-of-$6$ properties.

\item Anodyne morphisms are trivial cofibrations. Indeed they are cofibrations by definition and they are equivalences in the sense of definition \ref{def:Simplicial_weak_equivalence} by the first point of lemma \ref{lem:TrivFib_anodyne_Are_Eq_in_Delta+}.

\item As a consequence, strong fibrations are Kan fibrations.

\item Trivial fibrations, defined by the right lifting property against boundary inclusion, are both strong fibrations because they have the right lifting property against all cofibrations, and equivalence because of lemma \ref{lem:TrivFib_anodyne_Are_Eq_in_Delta+}.

\item A Kan fibration (or strong fibrations) with fibrant target is a trivial fibrations if and only if it is an equivalence (this follows from proposition \ref{prop:Simp_equiv_vs_SemiSimp_equiv} and the fact that this fact holds in weak model categories).

\end{itemize}
\end{remark}

Maybe it is a good point to recall the following very classical lemma that we will use constantly in this paper:

\begin{lemma} \label{lem:retract}
Assume that a map $f$ is factored as $f = p i$. If $i$ has the left lifting property against $f$, then $f$ is a retract of $p$. If $p$ has the right lifting property against $f$ then $f$ is a retract of $i$.
\end{lemma}

\begin{proof}
We only prove the first half of the claim, the second is just the dual statement. One form a morphism $h$ as the dotted diagonal filler in first square below  (obtained by the lifting property of $i$ against $f$), which can then be used to form a retract diagram:

\[
\begin{tikzcd}
A \ar[d,"i"] \ar[r,equal] & A \ar[d,"f"] \\ 
B \ar[r,"p"] \ar[ur,dotted,"h"{description}] & C \\
\end{tikzcd}
\qquad
\begin{tikzcd}
  A \ar[rr,bend left =40, equal] \ar[d,"f"] \ar[r,"i"] & B \ar[d,"p"] \ar[r,"h"] & A \ar[d,"f"]\\
  C \ar[r,equal] & C \ar[r,equal] & C \\
\end{tikzcd}
\]

\end{proof}

\begin{lemma} \label{Lem:Cof+eq=trivCof_and_=anodyne_if_fibranttarget}

\begin{enumerate}[label=(\roman*)]

\item[]

\item A cofibration is a trivial cofibration if and only if it has the left lifting property against all Kan fibrations between Kan complexes.

\item An arrow whose target is a Kan complex is a trivial cofibration if and only if it is anodyne.

\item An arrow whose target is a Kan complex is a strong fibration if and only if it is a Kan fibration.

\item A map is a trivial fibration if and only if it is a strong fibration and an equivalence.

\end{enumerate}

\end{lemma}

Because of the third point it is equivalent for a simplicial set $X$ that $X \rightarrow 1$ is a a Kan fibration (i.e. $X$ is a Kan complex) and that $X \rightarrow 1$ is a strong fibration. One will simply say that $X$ is fibrant.

\begin{proof}

\begin{enumerate}[label=(\roman*)]
\item[]
\item Let $f :A \hookrightarrow B$ be a cofibration that is also an equivalence, and we consider a lifting problem of $f$ against a Kan fibration between Kan complexes:

\[\begin{tikzcd}
A \ar[d,hook,"f"] \ar[r,"v"] & X \ar[d,two heads,"p"] \\
B \ar[r,"u"] & Y \\
\end{tikzcd}\]

In the special case where both $u$ and $v$ are equivalences, then by $2$-out-of-$3$, the map $p$ is also an equivalence. As it is a Kan fibration between Kan complexes it is also a trivial fibration, and hence the lifting problem has a solution because $f$ is a cofibration. We will now show that one can bring back the general case to this situation:

One can factor $u$ as an anodyne morphism followed by a Kan fibration: $B \overset{\sim}{\hookrightarrow} Y' \twoheadrightarrow Y$ and complete the diagram above by forming the pullback $P = Y' \times_Y X$ :

\[\begin{tikzcd}
A \ar[d,hook,"f"] \ar[r,"v'"] & P \ar[d,two heads] \ar[r,two heads] & X \ar[d,two heads] \\
B \ar[r,hook,"\sim"] & Y' \ar[r,two heads] & Y \\
\end{tikzcd}\]

The map $v'$ can be factorized as an anodyne morphism followed by a Kan fibration:

\[\begin{tikzcd}
A \ar[d,hook,"f"] \ar[r,"\sim"] & P' \ar[r,two heads] & P \ar[d,two heads] \ar[r,two heads] & X \ar[d,two heads] \\
B \ar[rr,hook,"\sim"] & & Y' \ar[r,two heads] & Y \\
\end{tikzcd}\]

The case treated above, where the two horizontal maps are equivalences, allows to produce a dotted diagonal lifting of the form:

\[\begin{tikzcd}
A \ar[d,hook,"f"] \ar[r,"\sim"] & P' \ar[dr,two heads] \ar[r,two heads] & P \ar[d,two heads] \ar[r,two heads] & X \ar[d,two heads] \\
B \ar[rr,hook,"\sim"] \ar[ur,dotted] & & Y' \ar[r,two heads] & Y \\
\end{tikzcd}\]

and this concludes the proof in the general case.

Conversely, assume $i:A \hookrightarrow B$ is a cofibration that has the left lifting property against all Kan fibrations between Kan complexes. One needs to show that $i$ is an equivalence. By taking an anodyne morphism $B \overset{\sim}{\hookrightarrow} B^f$ to a fibrant objects the composite $A \hookrightarrow B^f$ still has the announced lifting property so one can freely assume that $B$ is fibrant in order to show that $i$ is an equivalence. Under that assumption one factors $i$ as an anodyne morphism followed by a Kan fibration, the Kan fibration has a fibrant target so it has the right lifting property against $i$. Hence by the retract lemma \ref{lem:retract}, $i$ is a retract of the anodyne part of the factorization, hence anodyne itself and hence is an equivalence.

\item This second observation follows from last part of the proof of $(i)$ where we explicitly showed that a trivial cofibration with fibrant target is anodyne.

\item We have mentioned already that strong fibrations are Kan fibrations, and $(i)$ shows that Kan fibrations between Kan complexes are strong fibrations.

\item Trivial fibration have the right lifting property against all cofibrations, in particular against trivial cofibration hence they are strong fibration, and lemma \ref{lem:TrivFib_anodyne_Are_Eq_in_Delta+} shows they are equivalences. For the other direction, the proof is essentially the dual the proof of $(i)$. Let $p$ be a strong fibration that is also a weak equivalence, and consider a lifting problem of $p$ against a cofibration:

\[\begin{tikzcd}
 A \ar[d,hook] \ar[r] & X \ar[d,two heads,"p"] \\
B \ar[r] & Y \\
\end{tikzcd}\]

By factoring the map $A \rightarrow X$ into a cofibration $A \rightarrow A'$ followed by a trivial fibrations and taking the pushout of $A \hookrightarrow B$ along this map $A \rightarrow A'$ one reduces the problem to the case where the top map is an equivalence. One can then factor the bottom map as a cofibration followed a trivial fibration:

\[\begin{tikzcd}
A \ar[r,hook] \ar[dr,phantom,"\ulcorner"{very near end}] \ar[d,hook] & A' \ar[d,hook] \ar[rr,"\sim"] & & X \ar[d,two heads,"p"] \\
B \ar[r,hook] & B' \ar[r,hook] & Y' \ar[r,two heads,"\sim"] \ar[ur,dotted] & Y \\
\end{tikzcd}\]

where the dotted arrow exists because the composed cofibration $A \hookrightarrow Y'$ is a weak equivalence by the $2$-out-of-$3$ properties, and hence has the left lifting property against $p$. This provides a dotted filling for the initial square.

\end{enumerate}
\end{proof}

In order to conclude that one has a Model structure on simplicial sets, one needs one more proposition.

\begin{prop} \label{prop:Jiri_WFS}
Any morphism can be factored as a trivial cofibration followed by a strong fibration. 
\end{prop}

Again, we will show in \ref{prop:Kan=Strong} that this factorization system is actually the same as the anodyne/Kan fibrations factorizations system, i.e. that trivial cofibration are the same anodyne morphisms and that strong fibration are the same as Kan fibrations. Note that at this point it is immediate that anodyne morphism are trivial cofibrations, and hence that fibrations are Kan fibrations.

\begin{proof}

We will give two proof of this claim. The first one follows from \cite{makkai2014cellular}, more precisely its theorem $3.2$, which is not known to be constructive but allows to give a simple and direct proof of the present proposition.

In order to fix the issue with constructivity one gives a second, considerably less direct proof: as mentioned above in \ref{prop:Kan=Strong} we will prove independently of the present proposition that trivial cofibrations are the same as anodyne morphisms, hence showing that the weak factorization mentioned in the proposition exists and is simply the anodyne-Kan fibration weak factorization system (whose existence follows from the small object arguments).

We still give the first proof as we believe it is interesting on its own as it allows to construct the model structure on simplicial sets without needing to invoke Kan $\Ex^{\infty}$-functor.

Theorem 3.2 of \cite{makkai2014cellular} claims that the $2$-category of presentable categories endowed with a class of cellular morphisms generated by a set of morphisms is closed under pseudo-pullback, and that these pullback are constructed explicitly: the underlying category is the pullback of categories, and the class of cellular morphisms are the morphisms whose image in each component are in the specified classes. We apply this to the following square:

\[\begin{tikzcd}
  P \ar[r] \ar[d] \ar[dr,phantom,"\hspace{-10pt}\lrcorner"{description,very near start}] & (\text{Kan-Cplx},\text{TrivFib}) \ar[d] \\
(\widehat{\Delta},Cof) \ar[r] & (\text{Kan-Cplx},\text{All arrows}) \\
\end{tikzcd}\]

Where ``Cof'' denotes the class of cofibration in $\widehat{\Delta}$ which is generated by a set. Kan-Cplx denotes the category of ``algebraic Kan complexes'', i.e. simplicial set endowed with chosen lifting against horn inclusion and of morphisms compatible to these choices of lifting. The functor $\widehat{\Delta} \rightarrow \text{Kan-Cplx}$ send any simplicial set to the ``free algebraic Kan complexes it generates'',i.e. the left adjoint to the forgetful functor from algebraic Kan complex to simplicial set, or equivalently the functor sending a simplicial set to its canonical fibrant replacement as produced by R.Garner version of the small object argument.

The class $\text{TrivFib}$ is the left class of the weak factorization on Kan-Cplx cofibrantly generated by the image of the horn inclusion in $\widehat{\Delta}$. The right class of the weak factorization system are hence exactly the morphism whose image by the forgetful functor to $\widehat{\Delta}$ are Kan fibrations. It follows that the morphism in $\widehat{\Delta}$ which are sent to ``trivial cofibrations'' in Kan-Cplx are exactly the arrows that have the left lifting property against all Kan fibration between Kan complexes. Hence in this case the pullback is the category of simplicial sets with as set of cellular morphisms the map that are both cofibrations and have the left lifting property against Kan fibration between Kan complexes, i.e. the ``trivial cofibrations'' as defined above, hence this class of arrow is generated by a set, and hence by the small object argument it is one half of a weak factorization system.
\end{proof}

\begin{theorem} \label{th:Simplicial_Quillen_MS_1}

There is a model structure on the category of simplicial sets such that:

\begin{itemize}

\item The equivalences are as defined in \ref{def:Simplicial_weak_equivalence}.

\item The cofibrations and trivial fibrations are the same as in theorem \ref{main_theorem}.

\item The fibrations are the strong fibration of definition \ref{def:Simplicial_weak_equivalence}.

\end{itemize}

\end{theorem}

\begin{proof}
We have two weak factorization systems, trivial cofibrations have been defined as the cofibrations that are equivalences, and it was shown in \ref{Lem:Cof+eq=trivCof_and_=anodyne_if_fibranttarget} that trivial fibrations are the (strong) fibrations that are equivalences. Equivalences are stable by composition, satisfies $2$-out-of-$6$ and contains isomorphisms by definition, so this concludes the proof.
\end{proof}

\begin{prop} \label{prop:leftproper}

The model structure of theorem \ref{th:Simplicial_Quillen_MS_1} is left proper, i.e. the pushout of weak equivalence along a cofibration is a weak equivalence.

\end{prop}

\begin{proof}

Given a pushout square in the category of simplicial sets:

\[ \begin{tikzcd}
A \ar[d,hook] \ar[r,"\sim"] & C \ar[d,hook] \\
B \ar[r,"f"] & D \\  
\end{tikzcd}
\]

Then as the forgetful functor to semi-simplicial sets preserves all colimits, this square is again a pushout in the category of semi-simplicial sets. In this category every object is cofibrant, and pushout along a cofibration between cofibrant objects is a left Quillen functor hence preserves equivalences between cofibrant objects, hence $f$ is an equivalence in the category of semi-simplicial sets, and hence is an equivalence in $\widehat{\Delta}$ by definition (~\ref{def:Simplicial_weak_equivalence}).
\end{proof}

\section{Kan $\Ex^{\infty}$-functors}
\label{sec:KanEx}

The goal of this section is to introduce Kan's $\Ex$ and $\Ex^{\infty}$ functors and to use them in subsection \ref{subsec:AppOfExInf} to prove the remaining claim concerning the simplicial model structure. Most of the results here were (in their classical form) originally proved by Kan in \cite{kan1957css} (often with quite different proof that the ones we will provide here), but we will mostly follow the approach of S.Moss in \cite{moss2015another} which we will make constructive by only adding some details.

Subsection \ref{subsection:Degen_quo} is a preliminary section that is of some independent interest but which will have only a very marginal role in the paper: it will only be used to prove some decidability conditions (more precisely lemma~\ref{lem:Decidability_for_ExX}, which will be an easy consequence of \ref{lem:decidability_lifting_degenQuo} and proposition \ref{Prop:retractOfPosetAndDegeneracyQuotient}). As such it can be easily ignored by the reader.

Subsection \ref{subsec:PStructure} review the notion of ``P-structure'' introduced by S.Moss, which is mostly a language to talk more conveniently about ``Strongly anodyne morphisms'', i.e. transfinite composition of pushout of horn inclusion. This is a key tool to structure the proof of the main results of section \ref{subsec:KanEx_anodyne}.

Subsection \ref{subsec:KanExfunctorDef} introduce Kan's barycentric subdivision functor $\Sd$, its right adjoint $\Ex$ and Kan's $\Ex^{\infty}$ functor and proves some of their basic properties. This is very classical material that we reproduce here just for completeness and to discuss some constructive aspect.

Subsection \ref{subsec:KanEx_anodyne} reproduces (with some modification to make it constructive) S.Moss' proof in \cite{moss2015another} that the natural transformation $X \rightarrow \Ex^{\infty} X$ is an anodyne extension. Constructively this only works when $X$ is cofibrant. We also noted that S.Moss prove can be used to obtain a result which apparently was not known even classically: for any morphisms $f :X \rightarrow Y$ (with $X$ cofibrant) the natural morphisms:

 \[ X \rightarrow \Ex^{\infty} X \times_{\Ex^{\infty} Y} Y \]

is anodyne. This was known when $Y$ is terminal, or when $X \rightarrow Y$ is a fibration, and we will actually only use it in these two special cases.

Finally subsection \ref{subsec:AppOfExInf} uses the properties of this functor to conclude that all Kan fibrations are strong fibrations (proposition~\ref{prop:Kan=Strong}) and that the model structure on simplicial sets is indeed right proper (proposition \ref{prop:rightProperness}).

\subsection{Degeneracy quotient and questions of decidability} \label{subsection:Degen_quo}

In this section we establish some general results about a notion of ``degeneracy quotient'' that we will introduce. While the notion might have some interest on its own in other context its only use in the present paper is to prove some decidability results, which will follow from lemma \ref{lem:decidability_lifting_degenQuo} below. In fact, the only uses of this section in the present paper is in the proof of the decidability conditions of lemma \ref{lem:Decidability_for_ExX}. Proposition \ref{prop:DQ_stable_under_pullback} is not useful for the present paper, but will serve in some future work, in particular in \cite{Gambino2019simplicial} and it was more natural to include its proof here.

\begin{definition}

A morphism $f:X \rightarrow Y$ between simplicial sets is said to be \emph{degeneracy detecting} if:

\[ \forall x \in X, \quad f(x) \text{ is degenerated} \Rightarrow x \text{ is degenerated} \]

\end{definition}

Of course the converse implication is true for any simplicial map, so one has that $x$ is degenerated if and only if $f(x)$ is degenerated. One says that a cell $x \in X_n$ is $\sigma$-degenerated for some degeneracy $\sigma : [n] \rightarrow [m]$ if $x = \sigma^*y$ for some $y$. 

\begin{lemma}\label{lem:SigmaDegen_in_terms_of_degen}
Let $\sigma : [n] \rightarrow [m]$ be any degeneracy and $x \in X_n$ any cell. The following are equivalent:

\begin{enumerate}[label=(\roman*)]
\item $x$ is $\sigma$-degenerated.

\item For all face map $i:[k] \rightarrow [n]$ such that the composite $\sigma i$ is non-injective, the cell $i^* x$ is degenerated.
\end{enumerate}
\end{lemma}

\begin{proof}
If $x = \sigma^* y$ then for any such $i$, $i^*x= (\sigma i ) ^* y$ which is degenerated if $\sigma i$ is non-injective, so $(i) \Rightarrow (ii)$.

Conversely, let $x$ satisfying $(ii)$. If $\sigma$ is the identity the result is trivial. If $\sigma$ is not injective, then $x$ is in particular degenerated, i.e. there exist a non-trivial degeneracy $s: [n] \rightarrow [k]$ such that $x=s^*y$. Note that $y = d^*x$ for $d: [k] \rightarrow [n]$ any section of $s$. If for all section $d$ of $s$, $\sigma d$ is injective, then lemma \ref{lem_section_facto_Delta} below shows that $s$ factors as $j \sigma$ for some degeneracy $j:[m] \rightarrow [k]$ and $x=s^*y=\sigma^* j^* y$ is indeed $\sigma$-degenerated. If now $\sigma d$ is non-injective for some section $d$ of $s$, then $y = d^* x$ is degenerated by assumptions, hence one can write $x=s'^* y'$ for $y'$ of lower dimension than $x$ and start the argument above again, an induction on the dimension concludes the proof.
\end{proof}

\begin{lemma} \label{lem_section_facto_Delta}
Let $\sigma:[n] \rightarrow [m] $ and $s:[n] \rightarrow [k]$ be two degeneracy, assume that for all $d:[k] \rightarrow [n]$ a section of $s$, $\sigma d$ is injective, then there exists a (unique) $j:[m] \rightarrow [k]$ such that $s = j \sigma$.
\end{lemma}

One easily see it is also a necessary condition. 

\begin{proof}
One needs to show that, under the assumption of the lemma, for any two elements $i,j \in [n]$ if $\sigma i = \sigma j$ then $s i = s j$. If $s i \neq s j$, then there is a section $d$ of $s$ such that $ d s i = i$ and $d s j =j$ hence $\sigma j = (\sigma d) (s j)$  and $ \sigma i =(\sigma d) (s i)$, so the injectivity of $\sigma d$ implies that $\sigma i \neq \sigma j$. As equality in $[n]$ is decidable one can take the contraposite and concludes the proof.
\end{proof}

\begin{prop} Let $f :X \rightarrow Y$ be a map between simplicial sets, then the followings conditions are equivalents:
\begin{enumerate}[label=(\roman*)]
\item $f$ is degeneracy detecting.

\item If  $f(x)$ is $\sigma$-degenerated for some degeneracy $\sigma$ then $x$ is $\sigma$-degenerated as well.

\item $f$ has the (unique) right lifting property against all the degeneracy map $\Delta[n] \rightarrow \Delta[m]$.
\end{enumerate}
\end{prop}

\begin{proof}

$(ii)$ clearly implies the $(i)$ and the converse is immediate from lemma \ref{lem:SigmaDegen_in_terms_of_degen}. The lifting in $(iii)$ is automatically unique as degeneracy are epimorphisms in the presheaf category and this lifting property is a reformulation of $(ii)$.

\end{proof}

Given a simplicial set $X$, $x \in X([n])$ and $\sigma :[n] \rightarrow [m]$ a degeneracy, one defines $X[(x,\sigma)]$ as the pushout:

\[\begin{tikzcd}
\Delta[n] \ar[d,"\sigma"] \ar[r,"x"] & X \ar[d] \\
\Delta[m] \ar[r] & X[(x,\sigma)]
\end{tikzcd}\]

$X[(x,\sigma)]$ is the universal for map $X \rightarrow Y$ making $x$ ``$\sigma$-degenerated'', i.e. given a morphism $f:X \rightarrow Y$, it factors as $X \rightarrow X[(x,\sigma)]$ if and only if $f(x)= \sigma^*y$ for some $y \in Y([m])$, and such a factorization is unique when it exists.

More generally, given a collection $(x_i \in X([n_i]))_{i \in I}$ and $\sigma_i: [n_i] \rightarrow [m_i]$ one can define an object $X[(x_i,\sigma_i)]$ as the pushout of a coproduct of degeneracy maps, which has the following universal property: a morphism $f:X \rightarrow Y$ factors (uniquely) through $X \rightarrow X[(x_i,\sigma_i)] \rightarrow Y$ if and only if for all $i \in I$, $f(x_i)$ is $\sigma_i$-degenerated.

\begin{definition}

A morphisms is said to be a degeneracy quotient if it is obtain as $X \rightarrow X[(x_i,\sigma_i)]$ for some collection of $x_i \in X([n_i])$ and $\sigma_i:[n_i] \twoheadrightarrow [m_i]$ as above.

\end{definition}

\begin{prop} \label{prop:degenQuo_degenDetect_facto}

Degeneracy quotient and degeneracy detecting map form an orthogonal factorization system.

More precisely, for any morphisms $f: X \rightarrow Y$ its factorization is obtained as:

 \[ X \rightarrow X[(x_i,\sigma_i)] \rightarrow Y \]

where $(x_i,\sigma_i)$ is the collection of all $x_i$ and $\sigma_i$ such that $f(x_i)$ is $\sigma_i$-degenerated. 

\end{prop}

Note that this is essentially nothing more than the small object argument, though it is notable that in this case it converges in a single step.

\begin{proof}

It is clear from the universal property of $X[(x_i,\sigma_i)]$ that one has a factorization as in the lemma, and the first map is by definition a degeneracy quotient. The map $X[(x_i,\sigma_i)] \rightarrow Y$ is degeneracy detecting: given $x \in X[(x_i,\sigma_i)]$, it is the image of a $x_0 \in X$, if the image of $x$ is degenerated in $Y$ one has $f(x_0)= \sigma^* y$, hence $(x_0,\sigma)$ appears in the definition of $ X[(x_i,\sigma_i)]$, which forces the image of $x_0$, i.e. $x$, to be degenerated.

The orthogonality of the two class is relatively immediate as well: given a lifting problem:

\[\begin{tikzcd}
X \ar[r] \ar[d] & A \ar[d] \\
X[(x_i,\sigma_i) \ar[r] \ar[ur,dotted] & B \\
\end{tikzcd}\]

where the right map is degeneracy detecting, then a diagonal filling exists if and only the image of the $x_i$ in $A$ satisfies the appropriate degeneracy conditions. As their images in $B$ satisfies them because of the existence of the square, and as the map $A \rightarrow B$ is degeneracy detecting, this is immediate.

\end{proof}

The following is more or less a reformulation of what is a degeneracy quotient that will be convenient:

\begin{lemma} \label{CharacDegenLemma}

An epimorphism of simplicial set $p:A \rightarrow B$ is a degeneracy quotient if and only if for any map $f:A \rightarrow X$, the map $f$ factors through $p$ if and only if the following condition holds:

\begin{equation} \tag{D} \label{eq:degenCondition} \forall a \in A([n]) \quad p(a) \text{ is degenerated } \Rightarrow f(a) \text{ is degenerated.} \end{equation}

\end{lemma}

Note that if such a factorization exists then condition (\ref{eq:degenCondition}) holds without any assumption on $p$, so that if $p$ is a degeneracy quotient then a factorization exists if and only condition (\ref{eq:degenCondition}) holds.

\begin{proof}

It follows from \ref{lem:SigmaDegen_in_terms_of_degen}, that condition (\ref{eq:degenCondition}) is equivalent to:

\begin{equation} \label{eq:degenCondition_alt} \tag{D'}  \forall a \in A([n]) \quad p(a) \text{ is $\sigma$-degenerated } \Rightarrow f(a) \text{ is $\sigma$-degenerated.} \end{equation} 

A factorization of $f$ through $p$ is always unique as $p$ is an epimorphism, so saying that $f$ factors through $p$ if and only if condition (\ref{eq:degenCondition_alt}) (or (\ref{eq:degenCondition}) )  holds is equivalent to saying that $B$ (endowed with the map $p:A \rightarrow B$) has the universal property of $A[(a_i,\sigma_i)]$ where $(a_i,\sigma_i)$ are all the pairs of $a_i \in A([n])$ such that $p(a_i)$ is $\sigma_i$-degenerated. Hence this indeed holds if and only if $A \rightarrow B$ is a degeneracy quotient, as because of proposition \ref{prop:degenQuo_degenDetect_facto}, any degeneracy quotient $p: A \rightarrow B$ is isomorphic to $A \rightarrow  A[(a_i,\sigma_i)]$ where $(a_i,\sigma_i)$ are all the pairs of $a_i \in A([n])$ such that $p(a_i)$ is $\sigma_i$-degenerated.
\end{proof}

This observation has a quite interesting consequence that will be extremely useful to us, and in fact is the unique reason why we are interested in degeneracy quotient in the present paper:

\begin{lemma} \label{lem:decidability_lifting_degenQuo} Given $p:A \rightarrow B$ a degeneracy quotient of finite decidable simplicial sets, and $f :A \rightarrow X$ a morphisms to a cofibrant simplicial set. Then it is decidable if there exists a diagonal lift:

\[\begin{tikzcd}
A \ar[d] \ar[r] & X \\
B \ar[dotted,ur,"?"{description}]
\end{tikzcd} \]

\end{lemma}

\begin{proof}
One can use condition (\ref{eq:degenCondition}) of lemma \ref{CharacDegenLemma} to test whether such a diagonal lift exists. As $B$ is finite and decidable, degeneracy in $B$ is decidable. So for each cell $a \in A$ it is decidable if `` $p(a) \text{ is degenerated } \Rightarrow f(a) \text{ is degenerated}$'' as both side of the implication are decidable. Moreover this condition is automatically valid for all degenerated cells of $A$, so it is necessary to test it only on a finite number of cells to know whether $f$ factors through $p$, which makes the validity of condition (\ref{eq:degenCondition}) decidable and hence the existence of a diagonal lift decidable.
\end{proof}

The following lemma is obvious, but will be a convenient a technical tools to show that certain maps are degeneracy quotient:

\begin{lemma}\label{lem:test_degen_quotient} Let $p:A \rightarrow B$ be an epimorphism. One considers the equivalence relation $\sim_p $ on $A$ generated by:

\begin{itemize}

\item If $p(a)$ is $\sigma$-degenerated, then $a \sim_p \sigma^* t^* a$ for any section $t$ of $\sigma$.

\item $\sim_p$ is compatible with all the faces and degeneracy maps of $A$.

\end{itemize}

Then $p$ is a degeneracy quotient if and only if any two $a,a'\in A$ such that $pa = pa'$ one has $ a \sim_p a'$.

\end{lemma}

Note that for any morphisms, $a \sim_p a' \Rightarrow p a = pa'$.

\begin{proof}
One easily see that $\sim_p$ is exactly the simplicial equivalence relation by which one needs to quotient $A$ to obtain $A[(a_i,\sigma_i)]$ where $(a_i,\sigma_i)$ is the family of all $a_i$ such that $p(a_i)$ is $\sigma_i$ degenerated in $B$. By the second half of proposition \ref{prop:degenQuo_degenDetect_facto}, the map $p$ is a degeneracy quotient if and only if the second maps in the factorization $A \rightarrow A[(a_i,\sigma_i)] \rightarrow B$ is an isomorphism, which happens if and only if the relation $\sim_p$ is equivalent to $p(a) = p(a')$.
\end{proof}

We continue with a proposition that will be convenient to get examples of degeneracy quotient (see for example the proof of lemma \ref{lem:Decidability_for_ExX} for examples).

\begin{prop} \label{Prop:retractOfPosetAndDegeneracyQuotient}

Let $P$ be a poset with an idempotent order preserving endomorphism $\pi$ satisfying either $\forall x, \pi x \leqslant x$ or $\forall x, \pi x \geqslant x$. Let $Q = \pi P$. Then the morphisms between the simplicial nerve:

\[ N(P) \rightarrow N(Q) \]

induced by $\pi:P \rightarrow Q$ is a degeneracy quotient.

\end{prop}

\begin{proof}

We assume that $\pi x \leqslant x$. The other case follows by simply reversing the order relation on $P$ and on all objects of the category $\Delta$.

Let $p_0 \leqslant p_1 \leqslant \dots \leqslant p_n$ be an element of $N(P)_n$ and assumes that $p_0,\dots,p_{i-1} \in Q$, then one forms

\[ p_0 \leqslant p_1 \leqslant \dots \leqslant p_{i-1} \leqslant \pi p_{i} \leqslant p_{i} \leqslant \dots \leqslant p_n \]

It is an element of $N(P)_{n+1}$ whose image in $Q$ is degenerated as $\sigma^{i*} ( \pi p_0 \leqslant \dots \leqslant \pi p_n)$. This implies that in $N(P)$:

\[(p_0 \leqslant \dots  \leqslant p_n) \sim (p_0 \leqslant \dots \leqslant p_{i-1} \leqslant \pi p_{i} \leqslant p_{i+1} \leqslant \dots \leqslant p_n) \]

Hence using this for all $i$ from $0$ to $n$, one obtains that for any sequence $p_0 \leqslant \dots \leqslant p_n$ all the

\[ ( \pi p_0 \leqslant \dots \leqslant \pi p_{i-1} \leqslant p_{i} \leqslant \dots \leqslant p_n) \]

for $i=0,\dots, n+1$ are equivalent. In particular any sequence is equivalent to its image by $\pi$ and finally any two sequences whose image in $N(Q)$ are the same are equivalent.

\end{proof}

We finish with a proposition that will only be useful in future work (\cite{Gambino2019simplicial}):

\begin{prop} \label{prop:DQ_stable_under_pullback} The class of degeneracy quotient is stable under pullback. \end{prop}

\begin{proof} 
First we show that given a pullback of the form:

\[
\begin{tikzcd}
 P \ar[r] \ar[d,"\phi"{swap}] \ar[dr,phantom,"\lrcorner"{very near start}] & \Delta[n] \ar[d,"\sigma"] \\
  \Delta[k] \ar[r,"f"] & \Delta[m] \\
\end{tikzcd}
\]

where $\sigma$ is a degeneracy map, the map $\phi$ is a degeneracy quotient. This is proved using proposition \ref{Prop:retractOfPosetAndDegeneracyQuotient}. Indeed in such a pullback $P$ is nerve of the pullback of posets, that we will also denote $P$ (because the nerve functor commutes to pullback). We will show that the map $P \rightarrow [k]$ is of the form of proposition \ref{Prop:retractOfPosetAndDegeneracyQuotient}. The map $\sigma: [n] \twoheadrightarrow [m]$ is of this form, with the section $[n] \rightarrow [m]$ sending each $i \in [m]$ to the smallest element of the fiber, this gives an order preserving idempotent $\pi: [n] \rightarrow [n]$ such that $\pi x \leqslant x$. This induce an idempotent on $P$ sending a pair $(i,j)$ (with $i \in [k]$, $ j \in [n]$) to $\pi'(i,j) = (i,\pi j)$. This is still an element of $P$, $\pi'(i,j) \leqslant (i,j)$ it is idempotent, and its image identifies naturally with $[k]$.

Hence $\phi : P \rightarrow \Delta[k]$ is indeed a degeneracy quotient by proposition \ref{Prop:retractOfPosetAndDegeneracyQuotient}. We now show that given any pullback of the form:

\[
\begin{tikzcd}
 P \ar[r] \ar[d,"\phi"{swap}] \ar[dr,phantom,"\lrcorner"{very near start}] & \Delta[n] \ar[d,"\sigma"] \\
  X \ar[r,"f"] & \Delta[m] \\
\end{tikzcd}
\]

for a degeneracy $\sigma$, the map $\phi$ is a degeneracy quotient.

Indeed, one write:

\[ X = \colim_{\Delta[k] \rightarrow X} \Delta[k]\]

Given a $x: \Delta[k] \rightarrow X$ one write $P_x$ the pullback:

\[
\begin{tikzcd}
P_x \ar[d,"\phi_x"{swap}] \ar[r] \ar[dr,phantom,"\lrcorner"{very near start}] & P \ar[r] \ar[d,"\phi"{swap}] \ar[dr,phantom,"\lrcorner"{very near start}] & \Delta[n] \ar[d,"\sigma"] \\
\Delta[k] \ar[r] &  X \ar[r,"f"] & \Delta[m] \\
\end{tikzcd}
\]

All map $\phi_x$ are degeneracy quotient by the first part of the proof.  As the category of simplicial sets is a topos, colimits are universal, one has the morphism $\phi$ is the colimit of the arrows $\phi_x$ (in the category of arrows). As the class of degeneracy quotient is the left class of an orthogonal factorization system, the colimit $\phi$ is also a degeneracy quotient. To give an explicit argument: given a lifting problem of $\phi$ against a degeneracy detecting map one can construct for each $x$ a lifting:

\[
\begin{tikzcd}
P_x \ar[d,"\phi_x"] \ar[r] & P \ar[d,"\phi"] \ar[r] & A \ar[d] \\
\Delta[k] \ar[r] \ar[urr,dotted] & X \ar[r] & B \\  
\end{tikzcd}
\]

By uniqueness of the lifts, they will all be compatible and produces a morphisms from the colimits to $A$ making the square commutes.

Finally we can prove the claim in the proposition. Given a morphism $f:X \rightarrow Y$ any degeneracy map $\Delta[n] \rightarrow \Delta[m]$ over $Y$ (i.e with $\delta[m] \rightarrow Y$) is sent by the pullback functor $\widehat{\Delta}_{/Y} \rightarrow \widehat{\Delta}_{/X} $ to a degeneracy quotient. send degeneracy quotient to degeneracy quotient. But a general degeneracy quotient is pushout of coproduct of degeneracy map, this coproduct and pushout are preserved by the pullback functor (because the category of simplicial sets is cartesian closed), and coproduct of pushout of degeneracy are degeneracy quotient so this concludes the proof.
\end{proof}

\subsection{P-structures}
\label{subsec:PStructure}
This section recalls the notion of $P$-structure introduced in \cite{moss2015another} with some minor modification to make it more suitable to the constructive context. A ``P-structure'' on a morphism $f:A \rightarrow B$ is essentially a recipe for constructing it as an iterated pushout of (coproduct of) horn inclusion $\Lambda^i[n] \hookrightarrow \Delta[n]$. The general idea of this definition is that in such an iterated pushout cells are added by pairs: each pushout by a horn inclusion $\Lambda^i[n] \rightarrow \Delta[n] $ adds exactly two non-degenerate cells:

\begin{itemize}

\item[(\I)] The cell $P$ corresponding to the identity of $\Delta[n]$.

\item[(\II)] The cell $F$ corresponding to the the $i$-th face $\partial^i[n] : \Delta[n-1] \rightarrow \Delta[n]$.

\end{itemize}

These two cells are connected by $F= d_i P$. So if $A \hookrightarrow B$ is constructed by iterating such pushout, then one can partition the non-degenerate cells of $B$ that are not in $A$ into ``type $\I$'' and ``type $\II$'' and there should be a bijection which associate to any type $\II$ cell the type $\I$ cell that is added by the same pushout. The formal definition look like this:

\begin{definition} \label{def:P_structure}

Let $f:A \rightarrow B$ be a cofibration of simplicial sets. A $P$-structure on $f$ is the data of:

\begin{itemize}

\item A (decidable) partition of the set of non-degenerate cells of $B$ which are not in $A$ into:

\[ B_{\I} \coprod B_{\II} \]

called respectively type $\I$ cells and type $\II$ cells.

\item A bijection $P:B_\II \overset{\sim}{\rightarrow} B_{\I}$.

\end{itemize}

Such that:

\begin{enumerate}

\item For all $x \in B_{\II}$, $dim( Px) =dim(x)+1$

\item For all $x \in B_{\II}$, there is a unique $i$ such that $d_i( P x) = x$.

\item Every cell of $B_{\II}$ has finite $P$-height (see definition \ref{Def_P_height} and lemma \ref{lem:Finite_weak_Pheight_is_enough} below).

\end{enumerate}

\end{definition}

In \cite{moss2015another}, the last condition was formulated as a well-foundness condition. Well-foundness is a tricky notion constructively so we prefer to avoid it. It should be clear to the reader that the condition we will now explain is equivalent to well-foundness if one assumes classical logic, or if one has a nice enough notion of well-foundness constructively. Intuitively this last condition just assert that the ``recipe'' given by the $P$-structure to construct $B$ from $A$ as an iterated pushout of horn inclusion is indeed well-founded, i.e. can be executed. We will formulate it by introducing for each cell $b \in B$ a set:

\[ Ant(b) \]

of ``antecedent of $b$'' which corresponds to the set of cells that needs to be constructed before $b$ in the process described by $P$. In \cite{moss2015another} the well-foundness condition is essentially that the order relation generated by $b' \in Ant(b)$ is well-founded. As each $Ant(b)$ is a finite set this is equivalent to the fact that for each $b$ there is an integer $k$ such that when iterating $Ant(b)$ more than $k$ times one has only cells in $A$. This is this second definition that we will use in our constructive context.

\bigskip

More precisely:

Given a cell $b \in B_{\II}$ and let $i$ be the unique integer such that $d_i Px = x$, one defines the set $Ant(b)$ of antecedent of $b$ as:

\[ Ant_0(b) = \{ d_j P(b) | j \neq i \} \]

And one defines $Ant(b)$ as $Ant_0(b)$ together with all (iterated) faces of cells appearing in $Ant_0(b)$.

Similarly, if $b = P b'$ is type $\I$, one defines:

\[ Ant(b) = Ant(b')\]

Finally, if $b \in A$:

 \[ Ant(b) = \emptyset \]

and if $b$ is not in $A$ but degenerated, then

\[ Ant(b) = Ant(b') \]

where $b'$ is the unique non-degenerate cell such that $b= \sigma^* b'$.

One also defines $Ant_{\II}(b)$ to be the set of non-degenerate type $\II$ cell in $Ant_0(b)$. Note that in all cases $Ant(b)$ and $Ant_0(b)$ are Kurawtowski-finite\footnote{A set $X$ is said to be Kuratowski-finite if $\exists n, \exists x_1,\dots,x_n \in X, \forall x \in X, x= x_1 \text{ or } \dots \text{ or } x= x_n$.} sets, and as the subset of type $\II$ cell is decidable, $Ant_{\II}(b)$ is also Kurawtowski-finite. One defines $Ant^k(b)$ and $Ant_{\II}^k(b)$ by:

\[ Ant^1(b) = Ant (b) \qquad  Ant^k(b) = \bigcup_{c \in Ant b } Ant^{k-1} c \]
\[ Ant_{\II}^1(b) = Ant_{\II} (b) \qquad  Ant_{\II}^k(b) = \bigcup_{c \in Ant_{\II} b } Ant_{\II}^{k-1} c \]

Note that when applied to a non-degenerate type $\II$ cell $b \in B$, all elements of $Ant_{\II}(b)$  (and hence of $Ant^k_{\II}(b)$ as well) are non-degenerate type $\II$ cells of the same dimension as $b$.

\begin{definition} \label{Def_P_height}

\begin{itemize}
\item[]
\item One says that $b$ has finite $P$-height if there exists an integer $k$ such that:

\[ Ant^k (b) = \emptyset \]

\item One says that $b$ has finite weak $P$-height if there is an integer $k$ such that:

\[ Ant_{\II}^k (b) = \emptyset \]

\end{itemize}

\end{definition}

Note that for each given $k$ and $b \in B$, as the sets $Ant^k(b)$ and $Ant_{\II}^k(b)$ are Kuratowski-finite it is decidable whether or not $Ant^k(b)$ and $Ant_{\II}^k(b)$ are empty. In particular, assuming $b$ has finite (weak) $P$-height there is smallest integer $k$, called the (weak) $P$-height of $b$, such that $Ant_{(\II)}^k(b) = \emptyset$. But in general it might not be decidable whether $b$ has finite (weak) $P$-height or not.

\begin{lemma} \label{lem:Finite_weak_Pheight_is_enough}

Let $f :A \hookrightarrow B$ with a $P$-structure satisfying all the conditions of definition \ref{def:P_structure} but the last.  Then the following are equivalent:

\begin{itemize}

\item Every $b \in B$ has finite $P$-height.

\item Every non-degenerate type $\II$ cell $b \in B_{\II}$ has finite weak $P$-height.

\end{itemize}
\end{lemma}

\begin{proof}

It is clear that $Ant^k_{\II}(b) \subset Ant^k(b)$ hence the first condition implies the second. Conversely, assume that every $b \in B$ has finite weak $P$-height. We will prove by double induction on both the dimension and the weak $P$-height that all cells of $B$ have finite $P$-height.

First we assume that all cell of dimension $<n$ have finite $P$-height. 
Cells of $A$ have $P$-height zero. All cells of $B$ of dimension $n$ that are either degenerate or of type $\I$ satisfies $Ant(b) = Ant(b')$ for some $b'$ of dimension strictly less than $n$, hence for $b'$ of finite $P$-height by the induction assumption. As $Ant^k(b) = Ant^k(b')$ this implies that $b$ has finite $P$-height as well.

It remains to show that all non-degenerate $n$-cell of type $\II$ in $B$ have finite $P$-height. We do that by induction on their weak $P$-height.

Indeed for a general type $\II$ cell $b$, $Ant(b)$ is constituted of:

\begin{itemize}

\item Degenerate or type $\I$ cell, that are already known to have finite $P$-height.

\item Faces of cell in $Ant_0(b)$ which are hence of dimension $<n$ and hence known to be of finite $P$-height.

\item Non-degenerate type $\II$ cells that are hence elements of $Ant_{\II}(b)$, but 

\[\emptyset = Ant_{\II}^{k}(b) = \bigcup_{c \in Ant_{\II} b } Ant_{\II}^{k-1} c  \]

hence all $c \in Ant_{\II} b$ have weak $P$-height at most $k-1$, and hence they all have finite $P$-height by induction. 

\end{itemize}

So all elements of $Ant(b)$ have finite $P$-height, let $m$ be the maximum of all these $P$-height, one has that:

\[Ant^{m+1}(b) = \bigcup_{c \in Ant(b)} Ant^m(b) = \emptyset \]

\end{proof}

\begin{lemma} \label{lem:PstructurImpAnodyne}

A cofibration with a $P$-structure is anodyne. More precisely it is a $\omega$-transfinite composition of pushout of coproduct of horn inclusions.

\end{lemma}

A map will be called ``strongly anodyne'' if it admits a $P$-structure.

\begin{proof}
Let $A \hookrightarrow B$ be a cofibration with a $P$-structure.

Let $B_k \subset B$ be the subset of $B$ of cell of $P$-height at most $k$. One has $B_0 =A$, and $B_k$ is a sub-simplicial set. Indeed, for every cell $b \in B$ all faces of $b$ appears in $Ant(b)$ or are such that $Ant(d_i b)=Ant(b)$ and all degeneracies of $b$ satisfies $Ant(\sigma^* b)= Ant(b)$, hence they all have $P$-height at most $k$.

Let $U$ be the set of non-degenerate type $\II$ cell of $B$ of $P$-height exactly $k$. For each $u \in U$, let $i_u$ be the unique integer such that $d_{i_u} P(u)= u$.

Then the corresponding map $\Delta[n] \overset{P u}{\rightarrow} B_k$ send $\Lambda^{i_u}[n]$ to $B_{k-1}$ and both $u$ and $Pu$ are in $B_k- B_{k-1}$.

Hence taking the pushout:

\[\begin{tikzcd}
\Lambda^{i_u} [n] \ar[r] \ar[d] & \Delta [n] \ar[d] \\
B_{k-1} \ar[r] & R
\end{tikzcd}\]

produces the simplicial set $R \subset B_k$ whose cells are all those of $B_{k-1}$, $u$ and $Pu$ and all their degeneracy. Taking the pushout by the coproduct of all these horn inclusions for all $u \in U$ gives $B_{k-1} \rightarrow B_{k}$.

Hence $B= \bigcup B_k$ is a $\omega$-transfinite composition of the maps $B_k \rightarrow B_{k+1}$ which are all pushout of coproduct horn inclusion.

\end{proof}

Classically one also has the converse: any transfinite composition of pushouts of coproduct horn inclusion has a canonical $P$-structure. Constructively this sort of statement is somehow problematic, mostly because the general notion of ``transfinite composition'' require a notion of ordinal to be formulated appropriately, but it works perfectly fine if one restrict to $\omega$-composition:

\begin{prop}
The class of strongly anodyne morphism contains all horn inclusion and is stable under pushout and $\omega$-transfinite\footnote{Here the restriction to ``$\omega$'' is only to avoid the discussion of what is an ordinal constructively.} composition. Any morphism can be factored as a strongly anodyne morphisms followed by a Kan fibration, and any anodyne morphism is a retract of a strongly anodyne morphism.

\end{prop}
\begin{proof} Horn inclusion have a trivial $P$-structure with one cell of type $\I$ and one cell of type $\II$. It is easy to see that coproduct, pushout and transfinite composition of strongly anodyne map have $P$-structure induced by the $P$-structure we start from, for example if $A \hookrightarrow B$ has a $P$-structure, then $C \rightarrow B \coprod_A C$ has a $P$-structure where a cell in $B \coprod_A C$ is type $\I$ or $\II$ if and only if it is type $\I$ or $\II$ for the $P$-structure on $A \hookrightarrow B$ and the map $P$ is the same as the one on $B$, and similarly for coproduct and transfinite composition.

It follows that the factorization of the map as an anodyne followed by a Kan fibration obtained by the small object argument is a strongly anodyne morphism as it is constructed as a $\omega$-transfinite composition of pushout of coproduct of horn inclusion. Finally any anodyne morphism $j$ can be factored as a strongly anodyne morphism followed by a Kan fibration, and the usual retract lemma (\ref{lem:retract}) shows that $j$ is a retract of the strongly anodyne part of the factorization.

 \end{proof}

We finish this section by mentioning a very important example where this machinery applies, mostly to serve as an example to show how it can be used.

Given two morphisms $f:A\rightarrow B$ and $g:X \rightarrow Y$ between simplicial sets one define as usual $f \corner{\times} g$ the ``corner-product'' of $f$ and $g$ as the morphism:

\[f \corner{\times} g :  \left( A \times Y \right) \coprod_{A \times X} \left( B \times X \right) \rightarrow B \times Y\]

One then has the following well known proposition, which we have referred to in the introduction as the corner-product conditions, and which is a key point in establishing the existence of the weak model structure on simplicial sets. It also corresponds to the fact the model structure on simplicial sets that we are constructing is cartesian.

\begin{prop} \label{prop:pushoutProdCond}
If $i$ and $j$ are cofibrations, then $i \corner{\times} j$ is a cofibration as well. Is one of them is anodyne then $i \corner{\times} j$ is also anodyne.
\end{prop}

As usual (following for example the appendix of \cite{joyal2006quasi}) this implies the dual condition, that if $i:A \rightarrow B$ is a cofibration and $p:Y \rightarrow X$ is a fibration, then the map $[B,Y] \rightarrow [B,X] \times_{[A,X]} [A,Y]$ is a fibration (the brackets denotes the cartesian exponential in simplicial sets), and it is a trivial fibration as soon as either $i$ is anodyne or $p$ is a trivial fibration.

\begin{proof}
By usual abstract manipulation (see for example the appendix of \cite{joyal2006quasi}) it is sufficient to show it when $i$ and $j$ are generating cofibrations/generating anodyne map. If $i$ and $j$ are generating cofibrations it is very easy to check that $i \corner{\times} j$ is a cofibration as defined in the statement of our main theorem \ref{main_theorem}. It remains to check that if $i$ is one of the generating cofibrations, i.e. $\partial \Delta[n] \hookrightarrow \Delta[n]$ for some $n$, and $j$ is one of the generating anodyne morphisms, i.e. $\Lambda^k[m] \hookrightarrow \Delta[m]$ for some $k,m$, then $i \corner{\times} j$ is anodyne. This is done by constructing an explicit $P$-structure on $i \corner{\times} j$.

The first direct proof of this claim that we know of is in \cite{joyal2008notes} (theorem 3.2.2), here we follow the proof of S.Moss' in 2.12 of \cite{moss2015another} to show how P-structures works. We only treat the case $k < m $ for simplicity, by reversing the order relation on can treat the case $k>0$ similarly, which in particular cover the case $k=m$.

A $i$-cell of $\Delta[n] \times \Delta[m]$ is an order preserving function $[i] \rightarrow [n] \times [m]$. It is non-degenerate if and only if it is an injective function. The domain $D$ of $i \corner{\times} j$ is:

\[ \left( \Delta[n] \times \Lambda^k[m] \right) \coprod_{\partial \Delta[n] \times \Lambda^k[m]} \left( \partial \Delta[n] \times \Delta[m] \right) = \left( \Delta[n] \times \Lambda^k[m] \right) \bigcup \left( \partial \Delta[n] \times \Delta[m] \right) \]

It corresponds to the morphisms  $[i] \rightarrow [n] \times [m]$ such that either they skip a column or they skip a row other than $k$, where we consider that $[n] = \{0,\dots,n\}$ numbers the column of $[n] \times [m]$ and $[m]=\{0,\dots,k,\dots,m\}$ numbers the row. So the only non-degenerate cell of $\Delta[n] \times \Delta[m]$ that are not in $D$ are injection $[i] \rightarrow [n] \times [m]$ whose first projection takes all possible value, and whose second projection takes all possible values except maybe $k$.

One says that a cell is type $\II$ if either it skip the $k^{th}$ row by going directly from $(a,k-1)$ to $(a+1,k+1)$, in which case one define $Px$ by adding the intermediate step $(a,k-1),(a+1,k),(a+1,k+1)$ , or if the last point where the $k^{th}$ row is reached, is $(a,k)$ followed by $(a+1,k+1)$ in which case $Px$ is defined by inserting the intermediate step: $(a,k), (a,k+1), (a+1,k+1)$.

It is an easy exercise to check that this defines a $P$-structure.

\end{proof}

\subsection{Kan $\Ex$ and $\Sd$ functors} \label{subsec:KanExfunctorDef}

Consider the barycentric subdivision functor $\Delta \rightarrow \widehat{\Delta}$:

\[ \Delta[n] \mapsto  \Sd \Delta[n] := N \Kcal([n]) \]

Where $\Kcal([n])$ denotes the set of \emph{finite non-empty decidable} subsets of $[n]$. Functoriality in $[n]$ is given by direct image of subsets on $\Kcal[n])$.

This extend to an adjunction:

\[ \Sd : \widehat{\Delta} \leftrightarrows \widehat{\Delta} : \Ex  \]

with: 

\[ (\Ex X)_n = \Hom(\Sd \Delta [n], X) \qquad \Sd X = \colim_{\Delta [n] \rightarrow X} \Sd \Delta [n] \]

The barycentric subdivision construction has a nice expression not just for the $\Delta[n]$, but also for all objects which are in the image of the functor $\widehat{\Delta_+} \rightarrow \widehat{\Delta}$, indeed:

\begin{prop} \label{Prop:semi-simplicialSubdivision} The composite:

\[ \widehat{\Delta_+} \rightarrow \widehat{\Delta} \overset{\Sd}{\rightarrow} \widehat{\Delta} \]

Is the functor sending a semi-simplicial set $X$ to $N(\Delta_{+}/X)$.

\end{prop}

One can note that as the category $\Delta_+ /X$ is directed, the nerve $N(\Delta_{+}/X)$ is itself the image of the semi-simplicial set of its non-degenerate cells. We won't make any use of this remark though.

\begin{proof}
This functors $X \mapsto N(\Delta_{+}/X)$ preserves colimit, because it can be rewritten as:

\[ N(\Delta_+/X)_k = \coprod_{F: [k] \rightarrow \Delta_+} X(F(k)) \] 

which is levelwise a coproduct of colimits preserving functors.

Hence we are comparing to colimits preserving functor, so it is enough to show they are isomorphic when restricted to representable. But $\Delta_+/[n] \simeq \Kcal [n]$ functorially on map of $\Delta_+$ so this concludes the proof. 

\end{proof}

\begin{prop} \label{prop:Sd_Ex_preserves_cof_fib}$\Sd$ preserves cofibrations and anodyne morphisms, $\Ex$ preserves fibrations and trivial fibrations. \end{prop}

\begin{proof} It is enough to check that the image of the generating cofibrations and generating anodyne maps by $\Sd$ are cofibrations and anodyne respectively.

In both case one can use proposition \ref{Prop:semi-simplicialSubdivision} to computes $\Sd$ on the generators as they are image of semi-simplicial maps. This makes the results immediate for cofibrations:

\[ \Sd \partial \Delta[n] \rightarrow \Sd \Delta[n] \]

is the morphism $N( \Kcal[n] - \{ [n] \}) \rightarrow N( \Kcal[n])$ which is clearly a levelwise complemented monomorphisms between finite decidable, hence cofibrant, simplicial sets.

For anodyne:
\[ \Sd \Lambda^i[n] \rightarrow \Sd \Delta[n]\]

is the morphisms $N(\Kcal[n] - \{[n], [n]-\{i\} \}) \rightarrow N (\Kcal[n])$. It can then be checked completely explicitly that this is a (strongly) anodyne morphisms, see Proposition 2.14 of \cite{moss2015another} for an explicit description of a $P$-structure.

\end{proof}

There is a natural transformation:

\[ \Sd \Delta[n] \rightarrow \Delta [n] \]

Which is induced by the order preserving function:

\[ \max: \Kcal [n] \rightarrow [n] \]

sending each (decidable) subset of $[n]$ to its maximal element. By Kan extension, this gives us natural transformations:

\[ \Sd \overset{m}{\rightarrow} Id \qquad Id \overset{n}{\rightarrow} \Ex \]

One can hence define a sequences of functors:

\[\begin{tikzcd}[ampersand replacement=\&]
X \ar[r, "n_x"] \& \Ex X \ar[r,"n_{\Ex X}"] \&  \Ex^2 X \ar[r,"n_{\Ex^2 X}"] \& \dots \ar[r,"n_{\Ex^{k-1} X}"] \& \Ex^{k} X \ar[r,"n_{\Ex^k X}"] \& \dots \ar[r] \& \Ex^{\infty} X \\
\end{tikzcd}\]

with $\Ex^{\infty}$ the colimit.

\begin{lemma} \label{lemma:Ex_fibrant1} For each $k,n$, there is a (dotted) arrow $\Psi^k_n$ making the following triangle commutes. 

\[\begin{tikzcd}[ampersand replacement=\&]
\Sd^2 \Lambda^k [n] \arrow{dr}  \arrow{rr}{\Sd( m_{\Lambda^k [n]}) } \& \& \Sd \Lambda^k [n]\\
\& \Sd^2 \Delta[n] \arrow[dotted]{ur}{\Psi^k_n} \& 
\end{tikzcd}\]

\end{lemma}

\begin{proof} The proof given in \cite{cisinski2006prefaisceaux} as proposition 2.1.39 is purely combinatorial and constructive.



\end{proof}

\begin{cor} \label{cor:Ex_infty_fibrant}For every cofibrant simplicial set $X$, $\Ex^{\infty} X$ is a Kan complex. \end{cor}

The proof that follows essentially comes from \cite{cisinski2006prefaisceaux}. If one does not assume that $X$ is cofibrant it still applies to proves that $X$ has the ``existential'' right lifting property against horn inclusion, but it does not seems possible to give a uniform choice of solution to all lifting problems without this assumption. Without such a uniform choice of lifting against horn inclusion one cannot construct solution to lifting problems against more complicated anodyne morphism that involves an infinite number of pushout of horn inclusion, unless we assume the axiom of choice.

\begin{proof}

Lemma \ref{lemma:Ex_fibrant1} allows to show that given any solid diagram as below, there is a dotted filling:

\[\begin{tikzcd}[ampersand replacement=\&]
\Lambda^k [n] \arrow{r} \ar[d,hook] \& \Ex X \ar[d,"n_{\Ex X}"] \\
\Delta[n] \ar[r,dotted] \& \Ex^2 X 
\end{tikzcd}\]

Indeed, through the adjunction the map $\Lambda^k [n] \rightarrow \Ex X$ corresponds to an arrow $\Sd \Lambda^k [n] \rightarrow X$, which due to lemma \ref{lemma:Ex_fibrant1} can be extended in:

\[\begin{tikzcd}[ampersand replacement=\&]
\Sd^2 \Lambda^k [n] \ar[r,"\Sd m_{\Lambda^k[n]}"] \ar[d, "\Sd^2 \_"{swap}] \& \Sd \Lambda^k [n] \ar[r] \& X \\
\Sd^2 \Delta [n] \ar[ur, " \psi^k_n"{swap}] 
\end{tikzcd}\]

The resulting map $\Sd^2 \Delta[n] \rightarrow X$ corresponds to a map $\Delta[n] \rightarrow \Ex^2 X$ which has exactly the right property to make the square above commutes.

Now by smallness of $\Lambda^k[n]$, any map $\Lambda^k[n] \rightarrow \Ex^{\infty} X$ factors in $\Ex^{k} X$, the observation above produces a canonical filling in $\Delta[n] \rightarrow \Ex^{k+1} X$. The choice of the filling, seen as taking values in $\Ex^{\infty} X$, in general depends on $k$, but if one further assume that $X$ is cofibrant, than by lemma \ref{lem:Decidability_for_ExX}, the maps $\Ex^k X \rightarrow \Ex^{k+1} X$ are all level wise decidable inclusion, so there is a smallest $k$ such that the map $\Lambda^k[n] \rightarrow \Ex^{\infty} X$ factors into $\Ex^{k} X$ and this produces a canonical solution to the lifting problem.
\end{proof}

\begin{prop} \label{prop:Ex_infty_Preserves_Fib}
If $f:X \rightarrow Y$ is a fibration (resp. a trivial fibration) with $X$ and $Y$ cofibrant then $\Ex^{\infty} f :\Ex^{\infty} X \rightarrow \Ex^{\infty} Y$ is also a fibration (resp. a trivial fibration).

\end{prop}

Similarly to what happen with corollary \ref{cor:Ex_infty_fibrant}, without the assumption that $X$ and $Y$ are cofibrant it is only possible to obtain the ``existential'' form of the lifting property and no canonical choice of lifting.

\begin{proof}
Given a lifting problem:

\[\begin{tikzcd}
\Lambda^k[n] \ar[r] \ar[d,hook,"\sim"] & \Ex^{\infty} X  \ar[d] \\
\Delta[n] \ar[r] & \Ex^{\infty} Y 
\end{tikzcd}\]

There is an $i$ such that it factors into:

\[\begin{tikzcd}
\Lambda^k[n] \ar[r] \ar[d,hook,"\sim"] & \Ex^i X \ar[d] \ar[r] & \Ex^{\infty} X  \ar[d] \\
\Delta[n] \ar[r] & \Ex^i Y \ar[r] &\Ex^{\infty} Y 
\end{tikzcd}\]

Moreover, assuming $X$ and $Y$ are cofibrant, lemma \ref{lem:Decidability_for_ExX} shows that $\Ex^i X \subset \Ex^{i+1} X$ are levelwise decidable inclusion, so (by finiteness of $\Lambda^k[n]$ and $\Delta[n]$) the set of $i$ such that a factorization as above exists is decidable, and hence there is a smallest such $i$. Proposition \ref{prop:Sd_Ex_preserves_cof_fib} shows that $\Ex^i f$ is a fibration, so the first square has a diagonal lifting and this concludes the proof.

\end{proof}

\subsection{S.Moss' proof that $X \rightarrow \Ex X$ is anodyne} \label{subsec:KanEx_anodyne}

Let $f:X \rightarrow Y$ be a simplicial morphisms. One has a square:

\[\begin{tikzcd}
X \ar[d] \ar[r] & Y \ar[d] \\
\Ex^{\infty} X \ar[r] & \Ex^{\infty} Y
\end{tikzcd}\]

 Our goal in this section is to show that when $X$ is cofibrant the induced map:

\[ X \rightarrow \Ex^{\infty} X \fprod_{Ex^{\infty} Y} Y \]

is a strong anodyne morphism. Note that if $Y=\Delta[0]$ is the terminal object, then $Ex^{\infty}(Y)=Y$  hence the statement above boils down to the fact that $X \rightarrow Ex^{\infty} X$ is a strong anodyne morphism. The idea to consider this morphisms comes form D.C Cisinski's book \cite[Cor 2.1.32]{cisinski2006prefaisceaux}, but the proof below follows closely the proof given by S.Moss in \cite{moss2015another} that $X \rightarrow \Ex^{\infty} X$ is strong anodyne.

Following the argument given in \cite[Cor 2.1.32]{cisinski2006prefaisceaux} (reproduced in the proof of corollary \ref{cor:X->ExinftyX_Anodyne} below), it will be enough to show:

\begin{prop} \label{prop_XtoExyXisAnodyne}

Given $f :X \rightarrow Y$ a simplicial morphisms, with $X$ cofibrant, then the morphism:

\[ X \rightarrow \Ex X \fprod_{\Ex Y} Y \]

is strongly anodyne.

\end{prop} 

The proof will be concluded in \ref{proof_of_prop_XtoExyXisAnodyne}, essentially, we will construct an explicit $P$-structure on this map. This construction is mostly due to S.Moss in \cite{moss2015another}. In addition to the dependency in $Y$, the main new contributions of this paper in this section is to show that assuming $X$ is cofibrant one can show that sufficiently many decidability conditions can be proved to make S.Moss' argument constructive. In order to do that properly one needs to completely reproduce his argument.

Following, \cite{moss2015another} one introduces two functions between the $\Sd \Delta[n]$.

Let $j^k_n: \Sd \Delta[n] \rightarrow \Sd \Delta[n]$ and $r^k_n : \Sd \Delta [n+1] \rightarrow \Sd \Delta [n]$ be the maps defined at the level of posets by:

\[ j^k_n \{ i \} = \left\lbrace \begin{array}{l l}
\{i\} & \text{ if $i \leqslant k$ } \\
\{0,\dots,i\} & \text{ if $i >  k$ }\\
\end{array} \right. \qquad r^k_n \{ i \} = \left\lbrace \begin{array}{l l}
\{i\} & \text{ if $i \leqslant k$ } \\
\{0,\dots,i-1\} & \text{ if $i =  k+1$ }\\
\{ i-1 \} & \text{ if $i>k+1$}
\end{array} \right. \]

Both extended to non-singleton elements as binary join preserving maps. These functions satisfies a certain number of equations, we list here those that we will need, they are all due to S.Moss.

\begin{lemma} \label{lem:MossEq} 
\begin{align}
\label{eq:Moss0} j^k_n j^h_n & =   j^h_n j^k_n=j^h_n & & 0 \leqslant h \leqslant k \leqslant n \\
  Id_{\Delta[n]} & = r^k_n \circ \Sd \partial^{k+1}_{n+1} &  & 0 \leqslant k \leqslant n \label{eq:Moss1}\\
\label{eq:MossP1} j^k_n r^k_n & = (\Sd \sigma^k_n) j^k_{n+1} & & 0 \leqslant k \leqslant n \\
\label{eq:Moss5} j^h_n r^k_n &= j^h_n (\Sd \sigma^k_n) & & 0 \leqslant h < k \leqslant n \\
\label{eq:Moss8} r^k_n j^h_{n+1} & = j^h_n r^k_n & & 0 \leqslant h \leqslant k \leqslant n \\
\label{eq:Moss3} r^k_n (\Sd \partial^{i+1}_{n+1}) & = (\Sd \partial^{i}_n) r^k_{n-1} & & 0 \leqslant k < i \leqslant n \\
\label{eq:Moss7} j^k_n r^k_n r^k_{n+1} & = j^k_n r^k_n (\Sd \sigma^{k+1}_{n+1}) & & 0 \leqslant k \leqslant n \\
\label{eq:Moss4} j^k_{n+1} (\Sd \partial^{h}_{n+1}) j^k_n & = j^k_{n+1} (\Sd \partial^{h}_{n+1}) & & 0 \leqslant k \leqslant n \text{ and } 0 \leqslant h \leqslant n+1  \\
\label{eq:Moss2} j^k_n r^k_n (\Sd \partial^i_{n+1}) j^{k-1}_n &= j^k_n r^k_n (\Sd \partial^i_{n+1}) & &  0 \leqslant i \leqslant k \leqslant n \\
\label{eq:Moss9} (\Sd \sigma^h_{n}) j^k_{n+1} r^k_{n+1} &= j^{k-1}_n r^{k-1}_n (\Sd \sigma^h_{n+1}) & & 0 \leqslant h < k \leqslant n+1 \\
\label{eq:Moss10} (\Sd \sigma^h_{n}) j^k_{n+1} r^k_{n+1}& = j^{k}_n r^{k}_n (\Sd \sigma^{h+1}_{n+1}) & & 0 \leqslant k \leqslant h \leqslant n
\end{align}
\end{lemma}

\begin{proof}All the functions involved are nerve of join preserving maps between the $\Kcal[n]$, so it is enough to check the relations at the level of posets and when function are evaluated at $\{i\}$, where one has explicit formula for all of them. \end{proof}

As functions between the $\Sd \Delta[n]$, $j^k_n$ and $r^k_n$ automatically acts one the cells of $\Ex X$. One denotes this action by $x \mapsto x j^k_n$ and $x \mapsto x r^k_n$ which is compatible to the identification of cells of $\Ex X$ with functions $\Sd \Delta[n] \rightarrow x$.

By equation (\ref{eq:Moss0}), the $j^{k}_n$ are an increasing family of commuting projection whose image defines a series of subsets:

\[X_n =  J^0_n \subset J^1_n \subset \dots J^n_n = (\Ex X)_n \]

where the identifications with $(\Ex X)_n$ and $X_n$ comes from the fact that $j^n_n$ is the identity, and $j^0_n : \Kcal [n] \rightarrow \Kcal [n]$ has image isomorphic to $[n]$, with $j^0_n : \Kcal [n] \rightarrow [n]$ being the ``Max'' function used in the definition of the natural transformation $\Sd \Delta[n] \rightarrow \Delta[n]$.

We define:

 \[ \Ex_Y(X) = \Ex X \fprod_{\Ex Y} Y \]

An $n$-cell in $\Ex_Y$ is a morphism $\Sd \Delta[n] \rightarrow X$ whose image in $Y$ factors through the map $\Sd \Delta[n] \rightarrow \Delta[n]$. I.e. it is an $n$-cell of $x \in (\Ex X)_n$ which satisfies:

\[ f x j^0_n = f x \]

Note that because of relation (\ref{eq:Moss0}) and (\ref{eq:Moss8}), $\Ex_Y X$ is stable under the action of $j^k_n$ and $r^k_n$.

Before going any further, one needs to state some decidability conditions:

\begin{lemma}\label{lem:Decidability_for_ExX} If $X$ is a cofibrant simplicial set, then:

\begin{enumerate}

\item The inclusion $X \subset \Ex_Y X$ is levelwise decidable.

\item $\Ex_Y X$ is cofibrant and $X \rightarrow \Ex_Y X$ is a cofibration.

\item The sets $J^n_k \subset (\Ex_Y X)_n$ are decidable.

\end{enumerate}

\end{lemma}

\begin{proof} All these decidability problems corresponds to the decidability of a factorization of a map $\Sd \Delta[n] \rightarrow X$ through some epimorphism $\Sd \Delta[n] \rightarrow K$. In all this case we will show that the corresponding epimorphism is a degeneracy quotient using lemma \ref{Prop:retractOfPosetAndDegeneracyQuotient} and conclude about the decidability using lemma \ref{lem:decidability_lifting_degenQuo}.

\begin{enumerate}

\item It corresponds to the map $\Sd \Delta[n] \rightarrow \Delta[n] $ which is the nerve of the max function $\Kcal[n] \rightarrow [n]$, whose section $i \mapsto \{0,\dots,i\}$ satisfies the condition of lemma \ref{Prop:retractOfPosetAndDegeneracyQuotient}.

\item  One just needs to check degeneracy are decidable in $\Ex X$, so it is about the epimorphism $\Sd (\sigma) : \Sd \Delta[n] \rightarrow \Sd \Delta[m]$ for any degeneracy $\sigma$. It is the nerve of $\sigma: \Kcal [n] \rightarrow \Kcal[m]$ which has a section satisfying the condition of lemma \ref{Prop:retractOfPosetAndDegeneracyQuotient} which send every $P \in \Kcal[m]$ to $\sigma^{-1}P$

\item It corresponds to the map $j^k_n: \Sd \Delta[n] \rightarrow j^k_n (\Sd \Delta[n])$, which is just is the nerve of the projection $j^k_n : \Kcal [n] \rightarrow j^k_n \Kcal [n]$ which is already of the form of lemma \ref{Prop:retractOfPosetAndDegeneracyQuotient}.

\end{enumerate}

\end{proof}

We can now give the definition of the $P$-structure on $X \hookrightarrow \Ex_Y X$.

\begin{itemize}

\item Type $\I$ cells are the non-degenerated cells $v \in \Ex_Y(X)$ which are not\footnote{It appears that because of point \ref{Lem:KeyToExUnitPstructure:xPxSameLevel} of lemma \ref{Lem:KeyToExUnitPstructure} and the fact that $r^0_n$ is the same as $\Sd \sigma_0$ it is actually a consequence from the rest of the definition that type $\I$ cells are not in $X$.} in $X$  and can be written as $y r^k_n$ with $y \in J^k_n \subset \Ex_Y X$.

\item Point \ref{Lem:KeyToExUnitPstructure:TypeI=P_is_degen} of lemma \ref{Lem:KeyToExUnitPstructure} will prove that being type $\I$ is decidable. Type $\II$ cells are just the cells that are not of type $\I$ (and which are non-degenerated and not in $X$).

\item For any cell $x$ one defines $P x$ as $x r^k_n$ where $k$ is the smallest integer such that $x \in J^k_n$, i.e. $x \in J^k_n - J^{k-1}_n$. Lemma \ref{lem:Decidability_for_ExX} shows that the $J^k_n$ are decidable so there is indeed such a smaller integer $k$.

\end{itemize}

In order to show that being type $\I$ is decidable and that $P$ defined this way defines a bijection from type $\II$ cells to type $\I$ cells, one needs a few technical lemma that we have regrouped in:

\begin{lemma} \label{Lem:KeyToExUnitPstructure}

\begin{enumerate}

\item[]

\item \label{Lem:KeyToExUnitPstructure:d_k+1P x =x} If $ x \in J^k_n - J^{k-1}_n$, then $d_{k+1} P x =x$.

\item \label{Lem:KeyToExUnitPstructure:xPxSameLevel} $x \in J^k_n$ if and only if $P x \in J^k_{n+1}$

\item \label{Lem:KeyToExUnitPstructure:jk-1rk_degenerated} If $x \in J^{k-1}_n$ then $x r^k_n$ is degenerate.

\item \label{Lem:KeyToExUnitPstructure:P2Degenerated} $P^2 x$ is always degenerated.

\item \label{Lem:KeyToExUnitPstructure:PTIouDegenIsdegen} If $x$ is degenerated or type $\I$ or in $X$, then $Px$ is degenerated.

\item \label{Lem:KeyToExUnitPstructure:FacesofPx_smaller_than_k} If $x \in J^k_n - J^{k-1}_n$ then for all $i \leqslant k$ $d_i (P x) \in J^{k-1}_n$.

\item \label{Lem:KeyToExUnitPstructure:FacesofPx_bigger_than_k+1} If  $x \in J^k_n - J^{k-1}_n$  then for all $i$, with $k+1 <i \leqslant n+1$,  $d_i (Px)$ is either of type $\I$ or degenerated.

\item \label{Lem:KeyToExUnitPstructure:TypeI=P_is_degen} A non-degenerated cell $x$ in $(\Ex_Y X)_n - X_n$ is type $\I$ if and only $Px$ is degenerated.

\end{enumerate}

\end{lemma}

\begin{proof}

\begin{enumerate}

\item $d_{k+1} P x$ is $x  r^k_n  (\Sd \partial^{k+1})$ which is equal to $x$ by equation \eqref{eq:Moss1}.

\item Let $k$ is the smallest value such that $x j^k_n =x$, i.e.$ P x = x r^k_n$. Equation \eqref{eq:Moss8} gives $x r^k_n j^k_{n+1} =x j^k_n r^k_n =x r^k_n$. Hence $P x \in J^k_{n+1}$, in particular $x \in J^h_n \Rightarrow k \leqslant h \Rightarrow Px \in J^h_{n+1}$. Conversely, if $P x \in J^k_{n+1}$ then:

\[\begin{array}{r c l r}
x j^k_n & = & (Px) (\Sd \partial^{h+1}) j^k_n & \text{(as $x=d_{h+1} Px$)} \\
 & =&  (Px) j^k_{n+1} (\Sd \partial^{h+1}) j^k_n & \text{( as $P x \in J^k_{n+1}$)} \\
 & = &  (Px) j^k_{n+1} (\Sd \partial^{h+1}) & \text{ (by equation \eqref{eq:Moss4})} \\
 & = & x &\text{( $Px \in J^k_{n+1}$ and $x=d_{h+1}Px$)}\\
\end{array} \]

Hence $x \in J^k_n$.

\item $x r^k_n= x j^{k-1}_n r^k_n$ is degenerated because of equation \eqref{eq:Moss5}

\item Let $k$ such that $x \in J^k_n - J^{k-1}_{n}$, then $Px = x r^k_n = x j^k_n r^k_n$ and $Px \in J^k_{n+1} - J^{k-1}_{n+1}$ because of point (\ref{Lem:KeyToExUnitPstructure:xPxSameLevel}), hence $P^2 x = x r^k_n r^k_{n+1} = x j^k_n r^k_n r^k_{n+1} $ which is degenerated because of equation \eqref{eq:Moss7}.

\item  Equation \eqref{eq:Moss9} and \eqref{eq:Moss10} show that if $x$ is degenerated then $Px$ is degenerated. If $x \in X$, i.e. $x \in J^0_n$ then $Px = xr^0_n$ but $r^0_n = \Sd \sigma_0$ so $Px$ is degenerated.

It follows that if $x$ is of type $\I$, then $x = y r^k_n$ with $y \in J^k_n$ if $y \in J^{k-1}_n$ then $x$ is degenerated because of point (\ref{Lem:KeyToExUnitPstructure:jk-1rk_degenerated}) hence $Px$ is degenerated because of the first part of the present point, if $ y \notin J^{k-1}_n$ then $x = Py$ and hence $Px$ is degenerated because of point (\ref{Lem:KeyToExUnitPstructure:P2Degenerated}).

\item This follows immediately from equation \eqref{eq:Moss2} as $d_i(Px)=x j^k_n r^k_n (\Sd \partial^i)$.

\item For $k+1< i \leqslant n+1$ on has:

\[\begin{array}{r c l r}
j^k_n r^k_n (\Sd \partial^i_{n+1}) &= &  j^k_n (\Sd \partial^{i-1}_n) r^k_{n-1}  & \text{by equation \eqref{eq:Moss3}}\\
&=& j^k_n (\Sd \partial^{i-1}_n) j^k_{n-1} r^k_{n-1}& \text{by equation \eqref{eq:Moss4}}
\end{array}\]

This equations shows that for $x \in J^k_{n}$,  $d_i Px$ is of the form $y r^k_{n-1}$ for $y \in J^k_{n-1}$, namely $y=x (\Sd \partial^{i-1}) j^k_{n-1}$,  hence, if $d_i Px$ is non-degenerated, it is of type $\I$. 

\item We have shown in \ref{Lem:KeyToExUnitPstructure:PTIouDegenIsdegen} that if $x$ is type $\I$ then $Px$ is degenerated. Conversely let $x$ be a non-degenerated cell such that $Px$ is degenerated. Let $k$ be such that $x \in J^k_n - J^{k-1}_n$. One has $x = d_{k+1} Px$ by point \ref{Lem:KeyToExUnitPstructure:d_k+1P x =x} of the lemma, hence $d_{k+1} Px$ is non-degenerated, which means that $Px$ can only be $\sigma_k$-degenerated or $\sigma_{k+1}$-degenerated (otherwise $d_{k+1}PX$ would also be degenerated). If $Px$ is $\sigma_k$-degenerated then $d_k Px = d_{k+1} Px = x$, but by point \ref{Lem:KeyToExUnitPstructure:FacesofPx_smaller_than_k} of the lemma $d_k Px \in J^{k-1}_n$ so this is impossible. If $Px$ is $\sigma_{k+1}$-degenerated then $d_{k+2} Px = d_{k+1}Px = x$ hence point \ref{Lem:KeyToExUnitPstructure:FacesofPx_bigger_than_k+1} shows that $x$ is of type $\I$.

\end{enumerate}

\end{proof}

\begin{npar} \label{proof_of_prop_XtoExyXisAnodyne} We are now ready to prove proposition \ref{prop_XtoExyXisAnodyne}:
\begin{proof}

The goal is to show that the type $\I$ cell and the operation $P$ we have defined satisfies the condition of \ref{def:P_structure}, so that the map is anodyne because of \ref{lem:PstructurImpAnodyne}.

Point (\ref{Lem:KeyToExUnitPstructure:TypeI=P_is_degen}) of lemma \ref{Lem:KeyToExUnitPstructure} (combined with lemma \ref{lem:Decidability_for_ExX}) shows that being a type $\I$ cell is decidable. So one can indeed defines type $\II$ cells as the cells that are not of type $\I$ (and non-degenerate nor in the domain) and get a partition of the non-degenerate cells. It also follows from point  (\ref{Lem:KeyToExUnitPstructure:TypeI=P_is_degen}) that if $x$ is a type $\II$ cell then $Px$ is a non-degenerate cell, and it is type $\I$ (either by definition or because of point (\ref{Lem:KeyToExUnitPstructure:P2Degenerated}) ).
Finally, point (\ref{Lem:KeyToExUnitPstructure:xPxSameLevel}) show that $P$ preserve the $k$ such that $x \in J^k_n$, as $X \subset \Ex_Y X$ corresponds to $J^0_n$ it shows that $P$ never send cell to cell in $X$. So $P$ restricts into a function from type $\II$ cells to type $\I$ cells.

We now show that it is a bijection:

If $x$ is a type $\I$ cell than it can be written as $y r^k_n$ with $y \in J^k_n$. 
By point (\ref{Lem:KeyToExUnitPstructure:jk-1rk_degenerated}) of lemma \ref{Lem:KeyToExUnitPstructure}, if $y \in J^{k-1}_n$, then $x=y r^k_n$ is degenerated, hence $y \notin J^{k-1}_n$ and hence $x=Py$. By point (\ref{Lem:KeyToExUnitPstructure:PTIouDegenIsdegen}) of lemma \ref{Lem:KeyToExUnitPstructure} if $y$ is degenerated or type $\I$ then $x=Py$ is degenerated, hence $y$ is a type $\II$ cell. This proves the surjectivity of $P$.

If $x$ is a type $\II$ cell and $y= Px$, then $x=d_{k+1} Px$ (because of point \ref{Lem:KeyToExUnitPstructure:d_k+1P x =x} of lemma \ref{Lem:KeyToExUnitPstructure}) where $k$ can be characterized as the unique integer such that $y \in J^k_{n+1}-J^{k-1}_{n+1}$ (because of point \ref{Lem:KeyToExUnitPstructure:xPxSameLevel} of lemma \ref{Lem:KeyToExUnitPstructure}). Hence $P$ is injective on type $\II$ cell and this concludes the proof that $P$ is a bijection between non-degenerated type $\II$ cells and non-degenerated type $\I$ cells.

Finally if $x$ is a non-degenerate type $\II$ cell, and let $k$ such that $x \in J^k_n -J^{k-1}_n$. Point (\ref{Lem:KeyToExUnitPstructure:d_k+1P x =x}) of lemma \ref{Lem:KeyToExUnitPstructure} shows that $d_{k+1} (Px)=x$, while point (\ref{Lem:KeyToExUnitPstructure:FacesofPx_smaller_than_k}) and  (\ref{Lem:KeyToExUnitPstructure:FacesofPx_bigger_than_k+1}) shows that for all $i \neq k+1$, $d_i Px$ is either in $J^{k-1}_n$, type $\I$ or degenerated, hence always distinct from $x$. So there is indeed a unique $i$ such that $d_i P x =x$, and it is $k+1$.

It remains to proves the ``well-foundness'' or ``finite height'' condition. It follows from point (\ref{Lem:KeyToExUnitPstructure:FacesofPx_smaller_than_k}) and (\ref{Lem:KeyToExUnitPstructure:FacesofPx_bigger_than_k+1}) of lemma \ref{Lem:KeyToExUnitPstructure} that given $x \in J^k_n - J^{k-1}_n$ a non-degenerate type $\II$ cell, $Ant_{\II}(x) \subset J^{k-1}_n$. In particular, any cell $x \in J^k_n$ has weak $P$-height at most $k$, hence by lemma \ref{lem:Finite_weak_Pheight_is_enough} this shows that every cell has finite $P$-height and hence concludes the proof.

\end{proof}\end{npar}

\begin{cor} \label{cor:X->ExinftyX_Anodyne}
For any $f:X \rightarrow Y$ with $X$ cofibrant, the morphism:

\[ X \rightarrow \Ex^{\infty} X \fprod_{\Ex^{\infty} Y } Y   \]

Is strongly anodyne.

\end{cor}

\begin{proof}

Consider $\Ex^k X \fprod_{\Ex^k Y } Y \rightarrow Y$ and apply the functor $\Ex_Y$ to it. One obtains:

\[\begin{array}{ r c l} \Ex_Y  \left( \Ex^k X \fprod_{\Ex^k Y } Y \right) &=&  \Ex  \left( \Ex^k X \fprod_{\Ex^k Y } Y \right) \fprod_{\Ex Y} Y \\
 &=& \left( \Ex^{k+1}  X \fprod_{\Ex^{k+1} Y} \Ex Y \right) \fprod_{\Ex Y} Y
\end{array} \]

In the last terms the map from the term  $\left( \Ex^{k+1}  X \fprod_{\Ex^{k+1} Y} \Ex Y \right)$ to $\Ex Y$ used in the fiber product is just the second projection, so the fiber product simplifies to:

\[  \Ex_Y  \left( \Ex^k X \fprod_{\Ex^k Y } Y \right) = \Ex^{k+1} X \fprod_{\Ex^{k+1} Y } Y \]

And the natural map $ \Ex^k X \fprod_{\Ex^k Y } Y \rightarrow \Ex_Y  \left( \Ex^k X \fprod_{\Ex^k Y } Y \right) $ corresponds through this identification to just:

\[ n_{\Ex^k X} \fprod_{n_{\Ex^k Y}} Id_Y : \Ex^k X \fprod_{\Ex^k Y } Y \rightarrow \Ex^{k+1} X \fprod_{\Ex^{k+1} Y } Y \]

It follows by induction that the sequence of maps:

\[ X \rightarrow \Ex X \fprod_{\Ex Y } Y \rightarrow \dots \rightarrow \Ex^k X \fprod_{\Ex^k Y } Y \rightarrow \Ex^{k+1} X \fprod_{\Ex^{k+1} Y } Y \rightarrow \dots \]

are all strong anodyne maps (and all these objects are cofibrant), and the map $X \rightarrow \Ex^\infty X \fprod_{\Ex^\infty Y } Y$ is their transfinite composite (this last claim can either be observed very explicitly, or formally by commutation of directed colimits with finite limits).

\end{proof}

\subsection{Applications}
\label{subsec:AppOfExInf}

\begin{prop} \label{prop:Kan=Strong}
Kan fibration are the same as the strong fibrations of definition \ref{def:Simplicial_weak_equivalence}.
Dually, the trivial cofibrations of definition \ref{def:Simplicial_weak_equivalence}  are the same as anodyne morphisms.
\end{prop}

The proof given here, at least the case of a Kan fibration between cofibrant object, is essentially the proof proposition 2.1.41 of \cite{cisinski2006prefaisceaux}.

\begin{proof}
We start with the first half: we observed in \ref{rk:def_WE_fibrations} that strong fibrations are Kan fibrations. So we only need to show that any Kan fibration is a strong fibration. We first show this claim for $p: A \twoheadrightarrow B$ a Kan fibration between cofibrant object. One has that $\Ex^{\infty}(f)$ is a Kan fibration (by \ref{prop:Sd_Ex_preserves_cof_fib}) between fibrant objects (because of \ref{cor:Ex_infty_fibrant}), hence it is a strong fibration (by lemma \ref{Lem:Cof+eq=trivCof_and_=anodyne_if_fibranttarget}.(iii)), in particular any pullback of $\Ex^{\infty}(f)$ is also a strong fibration. This gives a factorization of $p$:

\[
\begin{tikzcd}
A \ar[r,hook,"\sim"] \ar[dr,"p"{swap}] & \Ex^{\infty}(A) \times_{\Ex^{\infty}(B)} B \ar[r] \ar[dr,phantom,"\lrcorner"{very near start}] \ar[d,->>] & \Ex^{\infty} A \ar[d,->>,"\Ex^{\infty }p"] \\
&B \ar[r] & \Ex^{\infty}B 
\end{tikzcd}\]

in an anodyne map (by corollary \ref{cor:X->ExinftyX_Anodyne}) followed by strong fibration as a pullback of the strong fibration $\Ex^{\infty}(p)$. So $p$ is a retract of the strong fibration part by the retract lemma (\ref{lem:retract}) and hence is itself a strong fibration.

We now move to the case of a general Kan fibration. We first show that a Kan fibration that is also an equivalence is a trivial fibration. Let $p:X \rightarrow Y$ be such a Kan fibration and weak equivalence, one needs to show that it has the right lifting property against all boundary inclusion: $\partial \Delta[n] \hookrightarrow \Delta[n]$, consider such a lifting problem:

\[\begin{tikzcd}
  \partial \Delta[n] \ar[d,hook] \ar[r] & X \ar[d,two heads,"f"] \\
\Delta[n] \ar[r] & Y \\
\end{tikzcd} \]

One first factors the map $\Delta[n] \rightarrow Y$ as a cofibration followed by a trivial fibration and we form a pullback of $f$ along the fibration part to get a diagram:

\[\begin{tikzcd}
  \partial \Delta[n] \ar[d,hook] \ar[r,"u"] & P \ar[r, two heads,"\sim"] \ar[d,two heads,"f'"] & X \ar[d,two heads,"f"] \\
\Delta[n] \ar[r,hook] & Z  \ar[r, two heads,"\sim"]  & Y \\
\end{tikzcd}\]

By $2$-out-of-$3$ the new fibration $f'$ is again a weak equivalence, but note that now the object $Z$ is cofibrant. One can further factor $u$ in a cofibration followed by a trivial fibration:

\[\begin{tikzcd}
  \partial \Delta[n] \ar[d,hook] \ar[r,hook]& K \ar[dr,two heads,"f''"{swap}] \ar[r,two heads,"\sim"] & P \ar[r, two heads,"\sim"] \ar[d,two heads,"f'"] & X \ar[d,two heads,"f"] \\
\Delta[n] \ar[rr,hook] & & Z  \ar[r, two heads,"\sim"]  & Y \\
\end{tikzcd}\]

$f''$ is a Kan fibration between cofibrant objects, hence is a strong fibration by the first part of the proof, moreover it is an equivalence hence it is a trivial fibration by the last point of lemma \ref{Lem:Cof+eq=trivCof_and_=anodyne_if_fibranttarget}, and hence it has the right lifting property against the boundary inclusion which show that the morphism $f$ is a trivial fibration as well.

One can then concludes the proof by the same argument as used in the proof of the first part of lemma \ref{Lem:Cof+eq=trivCof_and_=anodyne_if_fibranttarget}: Given a lifting problem of a trivial cofibration against a Kan fibration one can, using appropriate factorization, reduce to the case where the top and bottom map of the lifting square are weak equivalences, in which case the Kan fibration is a weak equivalence by $2$-out-of-$3$ and hence is a trivial fibration by the claim we just made, and hence has the right lifting property against all cofibration which concludes the proof.

For the second half of the proposition, given a trivial cofibration $j$ one factors it as an anodyne morphisms followed by a Kan fibration. By the first half of the proof the Kan fibration is a strong fibration and hence has the right lifting property against $j$. It immediately follows from the retract lemma \ref{lem:retract} that $j$ is a retract of the anodyne morphism and hence is anodyne it self.

\end{proof}

\begin{prop} \label{prop:rightProperness} The model structure of \ref{th:Simplicial_Quillen_MS_1} is right proper, i.e. the pullback of a weak equivalence along a fibration is again a fibration.
\end{prop}

\begin{proof}

We start with the case where all the objects in the pullback are cofibrant. This implies that the pullback itself is cofibrant because it is a subobject of the product which is cofibrant because of the cartesianess of the model structure \ref{prop:pushoutProdCond}, and the explicit description of cofibrant objects in terms of decidability of degeneratness of cell, immediately shows that a subobject of a cofibrant simplicial sets is cofibrant.

In this case, the result follows immediately from an application of Kan's $\Ex^{\infty}$ functor: It preserves the pullback square (because it is a right adjoint), it send each object to a fibrant object, when all the object are fibrant the result is true in any (weak) model category (a clearly constructive argument, valid in weak model category is given as corollary 2.4.4 in \cite{henry2018weakmodel}), and it detect equivalences between cofibrant objects because the morphism $X \rightarrow \Ex^{\infty} X$ is anodyne (hence an equivalence) for $X$ cofibrant.

It appears that having right properness when all the objects are cofibrant is sufficient to deduce the general case by taking cofibrant replacement of all the objects involved in the appropriate order: Given a pullback $P = B \times_A C$ one constructs cofibrant replacement of $B^c \overset{\sim}{\twoheadrightarrow} B, \dots $ which still form a diagram such that the comparison maps $ B^c \times_{A^c} C^c \rightarrow  B \times_A C$ is again a trivial fibrations. This is achieved by constructing first $A^c$ and then defining $B^c$ and $C^c$ respectively as cofibrant replacement of the pullbacks $B \times_A A^c$ and $C \times_A A^c$. Assuming moreover that $B \twoheadrightarrow A$ is a fibration one also obtains this way that $B^c \twoheadrightarrow A^c$ is a fibration. Once this is done one deduces immediately the result in the general case from the result for the cofibrant replacement.

\end{proof}

\bibliography{/home/henry/Documents/Latex/Biblio}{}
\bibliographystyle{plain}

\end{document}